\documentclass[11pt,leqno]{amsart}
\usepackage{amsmath,amssymb,amsthm}
\usepackage{hyperref,tikz-cd,graphicx}
\usepackage{dot2texi}

\usepackage{tikz}
\usetikzlibrary{arrows}
\usetikzlibrary{decorations.markings}
\tikzset{nimplies/.style={
		decoration={markings,
			mark= at position 0.5 with {
				\node[transform shape] (tempnode){\tiny/};
			}},postaction={decorate}}}
			
\tikzset{negated/.style={
        decoration={markings,

            mark= at position 0.5 with {
                \node[transform shape] (tempnode) {$\backslash$};
            }
        },
        postaction={decorate}
    }
}
\usepackage{float}
\usepackage{xcolor}
\usepackage[shortlabels]{enumitem}

\newtheorem{theorem}{Theorem}[section]
\newtheorem*{theorem*}{Theorem}
\newtheorem{lemma}[theorem]{Lemma}
\newtheorem{ques}[theorem]{Question}

\newtheorem{proposition}[theorem]{Proposition}
\newtheorem{corollary}[theorem]{Corollary}

\theoremstyle{definition}
\newtheorem{definition}[theorem]{Definition}
\newtheorem{example}[theorem]{Example}

\theoremstyle{remark}
\newtheorem{remark}[theorem]{Remark}
\numberwithin{equation}{section}

\def\ignora#1{}

\renewcommand{\geq}{\geqslant}
\renewcommand{\leq}{\leqslant}

\newcommand{\norm}[1]{\left\Vert#1\right\Vert}

\newcommand{\spn}{\operatorname{span}}

\newcommand{\eps}{\varepsilon}

\newcommand{\Xs}{X^{\ast}}
\newcommand{\xs}{x^{\ast}}

\newcommand{\lxy}{\mathcal{L}(X,Y)}

\newcommand{\iten}{\widehat{\otimes}_\varepsilon}
\newcommand{\pten}{\widehat{\otimes}_\pi}

\newcommand{\alcomment}[1]{{\color{blue}Aleksei: #1\color{black}}}

\newcommand{\e}{\varepsilon}

\newcommand{\n}{\norm}
\newcommand{\bb}{\mathbb}
\newcommand{\du}[1]{{#1}^*}

\begin{document}
\title[ASD2P for infinite cardinals]{Attaining strong diameter two property for infinite cardinals}	

\author {Stefano Ciaci, Johann Langemets, Aleksei Lissitsin} 
\address{Institute of Mathematics and Statistics, University of Tartu, Narva mnt 18, 51009 Tartu, Estonia}
\email{stefano.ciaci@ut.ee, johann.langemets@ut.ee, aleksei.lissitsin@ut.ee}
\thanks{This work was supported by the Estonian Research Council grant (PSG487).}
\urladdr{\url{https://johannlangemets.wordpress.com/}}

\begin{abstract}
    We extend the (attaining of) strong diameter two property to infinite cardinals. In particular, a Banach space has the 1-norming attaining strong diameter two property with respect to $\omega$ (1-ASD2P$_\omega$ for short) if every convex series of slices of the unit ball intersects the unit sphere. We characterize $C(K)$ spaces and $L_1(\mu)$ spaces having the 1-ASD2P$_\omega$. We establish dual implications between the 1-ASD2P$_\omega$, $\omega$-octahedral norms and Banach spaces failing the $(-1)$-ball-covering property. The stability of these new properties under direct sums and tensor products is also investigated. 
\end{abstract}

\maketitle

\section{Introduction}\label{sec: intro}

An important tool for studying the geometry of the unit ball in a Banach space was first observed by J.~Bourgain -- every non-empty relatively weakly open subset of the unit ball of a Banach space contains a finite convex combination of slices \cite[Lemma~II.1]{Ghoussoub1987b}. On the other hand, there are finite convex combinations of slices which fail to be relatively weakly open \cite[Remark~IV.5]{Ghoussoub1987b}. Although in general the converse of Bourgain's lemma fails, it was proved in \cite[Remark~IV.5]{Ghoussoub1987b} that it holds in the positive part of the unit sphere of $L_1[0,1]$ and in \cite[Theorem~2.3]{abrahamsen_relatively_2018} it was shown that it also holds in the unit ball of $C(K)$, whenever $K$ is a scattered compact space. The latter result was subsequently extended to spaces of the type $C_0(K,X)$  in \cite{abrahamsen_banach_2018} and to certain $C^\ast$-algebras in \cite{becerra_guerrero_relatively_2019}.

Recall that in an infinite-dimensional Banach space every non-empty relatively weakly open subset of the unit ball intersects the unit sphere. Therefore, the requirement that every convex combination of slices of the unit ball intersects the unit sphere is a weaker property compared to the converse of Bourgain's lemma. In \cite[Section~3]{abrahamsen_relatively_2018} it is wondered which Banach spaces satisfy this weaker condition. Motivated by this question, the property was studied under the name (P3) in \cite{haller_convex_2017}, (CS) in \cite{lopez-perez_strong_2019} and \emph{attaining strong diameter two property} (ASD2P for short) in the recent preprint \cite{lopez-perez_characterization_2020}. From now on we will use the name ASD2P. 

In \cite[Theorem~3.4]{lopez-perez_strong_2019} it is proved that a Banach space $X$ has the ASD2P if and only if for every finite convex combination of slices $C$ of $B_X$ there are $x,y\in C$ such that $\|x+y\|=2$. Therefore, the ASD2P implies that every finite convex combination of slices has diameter two. This property is called the \emph{strong diameter two property} (SD2P for short) and it was first introduced in \cite{abrahamsen_remarks_2013}. Note that there are spaces with the SD2P, but failing to have the ASD2P \cite[Example~3.3]{lopez-perez_strong_2019}. 

The main focus of this paper is to introduce and study the infinite analogues of the (A)SD2P and to provide examples and several properties of them. In particular, a Banach space $X$ has the 1-ASD2P$_\omega$ if, for every sequence $(S_i)_i$ of slices of its unit ball and for every sequence $(\lambda_i)_i\subset[0,1]$ satisfying $\sum_{i=1}^\infty\lambda_i=1$, one has that $\sum_{i=1}^\infty\lambda_iS_i$ intersects the unit sphere of $X$. For example, the spaces $\ell_\infty$, $C[0,1]$, $L_1[0,1]$,  and $L_\infty[0,1]$ all have the 1-ASD2P$_\omega$ (Corollaries~ \ref{cor: ell infty C[0,1] CS}, \ref{cor: L_infty has CS} and \ref{cor: L_1(mu) CS_omega}). On the other hand, the space $c_0$ has the ASD2P, but fails the 1-ASD2P$_\omega$ (Remark~\ref{rem: def ASD2P}(c)).

Let us now describe the organization of the paper. In Section~\ref{sec: prelim} we introduce some notation and equivalent criteria for the (A)SD2P, which then allow us to extend them to their corresponding infinite versions -- SD2P$_\kappa$ and 1-ASD2P$_\kappa$, where $\kappa$ is an infinite cardinal (Definitions~\ref{def: kappa-SD2P} and \ref{def: kappa-CS}). It is worth pointing out that while the extension of the ASD2P to the 1-ASD2P$_\omega$ is straightforward (namely, by replacing finite convex combinations with convex series of slices), the same approach doesn't produce any new concept for the SD2P (Lemma~\ref{lem: SD2P=infinite SD2P}). We use a different way to extend these properties which will be meaningful for uncountable cardinals too.

Section~\ref{sec: duality} is devoted to establishing the duality between these new diameter two properties, $\kappa$-octahedrality, the failure of the $(-1)$-BCP$_\kappa$ (Definition~\ref{def: kappa-OH}), and the Daugavet property. Recall that a Banach space $X$ is said to have the \emph{Daugavet property} if the equation 
$$\|I+T\|=1+\|T\|$$
holds for every rank-1 operator $T:X\rightarrow X$ (see \cite{werner_recent_2001} for a nice survey on the topic). In Sections~\ref{sec: direct sums} and \ref{sec: tensor products} we investigate the stability of the SD2P$_\kappa$ and 1-ASD2P$_\kappa$ under direct sums and tensor products, respectively. Finally, Section~\ref{sec: examples} is devoted to studying several examples of classes of spaces with and without the SD2P$_\kappa$ or the 1-ASD2P$_\kappa$.

We pass now to introduce some notation. We consider only real infinite-dimensional Banach spaces. For a Banach space $X$, $\Xs$ denotes its topological dual, $B_X$ its closed unit ball, $S_X$ its unit sphere and $\text{dens}(X)$ its density character. A \emph{slice} $S$ of the unit ball is a set of the form 
$$S=\{x\in B_X:x^*(x)\ge1-\varepsilon\},$$
where $x^*\in S_{X^*}$ and $\varepsilon\in (0,1)$. For a set $A$ and a cardinal $\kappa$, we denote by $\mathcal P(A)$ the power set of $A$ and by $\mathcal P_\kappa(A)$ the set of all subsets of $A$ of cardinality at most $\kappa$. We will often replace "$\kappa$" with the symbol "$<\kappa$" to represent its obvious analogue (e.g. $\mathcal P_{<\kappa}(A)$ is the set of all subsets of $A$ of cardinality strictly smaller than $\kappa$). Given two Banach spaces $X$ and $Y$, we denote by $p_X$ the projection from $X\times Y$ onto $X$ and by $i_X$ the isometrical embedding of $X$ into $X\times Y$.

\section{Preliminary results}\label{sec: prelim}

We begin by proving that a set $A$ in a normed space behaves in an additive way with respect to the norm if and only if we can find an element in the dual that attains its norm in every element of $A$.

\begin{lemma}\label{lem: properties additive sets}
    Let $X$ be a normed space, $n\in \mathbb{N}$ and $x_1,\ldots,x_n\in X$. If $\|\sum_{i=1}^n x_i\|=\sum_{i=1}^n\|x_i\|$, then the following hold:
    \begin{itemize}
        \item[(a)] $\|\sum_{i=1}^n t_ix_i\|=\sum_{i=1}^n t_i\|x_i\|$ for every $t_1,\ldots,t_n\ge0$;
        \item[(b)] $\|\sum_{i=1}^n t_ix_i\|\ge\sum_{i=1}^n t_i\|x_i\|$ for every $t_1,\ldots,t_n\in\mathbb R$.
    \end{itemize}
\end{lemma}
\begin{proof}
    (a). Suppose at first that $t_1,\ldots,t_n\in[0,1]$. If $\|\sum_{i=1}^n t_ix_i\|<\sum_{i=1}^n t_i\|x_i\|$, then
    $$\|\sum_{i=1}^n x_i\|\leq \|\sum_{i=1}^n t_ix_i\|+\|\sum_{i=1}^n (1-t_i)x_i\| <\sum_{i=1}^n t_i\|x_i\|+\sum_{i=1}^n (1-t_i)\|x_i\|=\sum_{i=1}^n\|x_i\|,$$
    which is a contradiction, hence the claim holds whenever $t_1,\ldots,t_n\in[0,1]$. Suppose now that $t_1,\ldots,t_n\ge0$. If there is some $t_j>1$, then call $t:=\max \{t_i\colon i\in\{1,\dots,n\}\}$ and notice that
    $$\|\sum_{i=1}^n t_ix_i\|=t\Big\|\sum_{i=1}^n \frac{t_i}{t}x_i\Big\|=t\sum_{i=1}^n\frac{t_i}{t}\|x_i\|=\sum_{i=1}^n t_i\|x_i\|.$$
    (b) follows from (a), indeed, if $t_1,\ldots,t_n\in\mathbb R$, then
    $$\|\sum_{i=1}^n t_ix_i\|\ge\|\sum_{\substack{i=1\\ t_i\ge0}}^n t_ix_i\|-\|\sum_{\substack{i=1\\ t_i<0}}^n t_ix_i\|\ge\sum_{\substack{i=1\\ t_i\ge0}}^n t_i\|x_i\|-\sum_{\substack{i=1\\ t_i<0}}^n |t_i|\|x_i\|=$$
    $$=\sum_{i=1}^n t_i\|x_i\|.$$
\end{proof}

\begin{proposition}\label{prop: norm attaining on norm additive}
    Let $X$ be a normed space and $A\subset X$. The following are equivalent:
    \begin{itemize}
        \item[(i)] There is $x^*\in X^*\setminus\{0\}$ such that $x^*(x)=\|x^*\|\|x\|$ for every $x\in A$;
        \item[(ii)] $\|\sum_{i=1}^{n} x_i\|=\sum_{i=1}^{n}\|x_i\|$ for every $x_1,\ldots,x_n\in A$.
    \end{itemize}
\end{proposition}
\begin{proof}
  (i)$\implies$(ii). $\|\sum_{i=1}^{n} x_i\|\ge\|x^*\|^{-1}x^*(\sum_{i=1}^{n} x_i)=\sum_{i=1}^{n}\|x_i\|$ holds for every $x_1,\ldots,x_n\in A$.
  
  (ii)$\implies$(i). Define $x^*:\spn (A)\rightarrow\mathbb R$ by $x^*(\sum_{i=1}^{n} t_i x_i):=\sum_{i=1}^{n} t_i\|x_i\|$. We claim that $x^*$ is well defined. Suppose that $\sum_{i=1}^{n} t_ix_i=\sum_{j=1}^{m} s_j y_j$ for some $t_1,\ldots,t_n,s_1,\ldots,s_m\in\mathbb R$ and $x_1,\ldots,x_n,y_1,\ldots,y_m\in A$. By Lemma~\ref{lem: properties additive sets}(b),
  $$0=\|\sum_{i=1}^{n} t_ix_i-\sum_{j=1}^{m} s_j y_j\|\ge\sum_{i=1}^{n} t_i\|x_i\|-\sum_{j=1}^{m} s_j\|y_j\|.$$
  We can switch the roles of $\sum_{i=1}^{n} t_ix_i$ and $\sum_{j=1}^{m} s_j y_j$, hence $\sum_{i=1}^{n} t_i\|x_i\|=\sum_{j=1}^{m} s_j\|y_j\|$, that is the claim. Notice that $x^*$ is linear and  Lemma~\ref{lem: properties additive sets}(b) implies that $\|x^*\|\le1$. By Hahn--Banach theorem we can extend $x^*$ to an element of $X^*$ with the desired properties.
\end{proof}

G.~Godefroy and B.~Maurey (see \cite{godefroy_metric_1989}) introduced octahedral norms in order to characterize the Banach spaces that contain a copy of $\ell_1$. Recall that a normed space $X$ is called \emph{octahedral} if one of the following equivalent conditions holds:
\begin{itemize}
    \item[(i)] For every finite-dimensional subspace $Y\subset X$ and $\varepsilon>0$ there exists $x\in S_X$ such that
    $$\|y+\lambda x\|\ge(1-\varepsilon)(\|y\|+|\lambda|)\text{ for every }y\in Y\text{ and }\lambda\in\mathbb R;$$
    \item[(ii)] For every $x_1,\ldots,x_n\in S_X$ and $\varepsilon>0$ there exists $x\in S_X$ such that
    $$\|x_i+x\|\ge2-\varepsilon\text{ for every }i\in\{1,\ldots,n\}.$$
\end{itemize}

In \cite{ciaci_characterization_2021}, the following extensions of octahedral norms and ball-covering properties (see \cite{Cheng2006, guirao_remarks_2019}) to infinite cardinals were considered.

\begin{definition}[see {\cite[Definition~2.3 and 5.3]{ciaci_characterization_2021}}]\label{def: kappa-OH}
    Let $X$ be a normed space and $\kappa$ an infinite cardinal. We say that $X$ is \emph{$\kappa$-octahedral} if for every subspace $Y\subset X$ such that $\text{dens}(Y)\le\kappa$ and $\varepsilon>0$ there is $x\in S_X$ such that
    \begin{equation}\label{eq: k-OH}
        \|y+\lambda x\|\ge(1-\varepsilon)(\|y\|+|\lambda|)\text{ for every }y\in Y\text{ and }\lambda\in\mathbb R.
    \end{equation}
    We say that $X$ \emph{fails the $(-1)$-BCP$_\kappa$} if in addition (\ref{eq: k-OH}) holds for $\varepsilon=0$.
\end{definition}

Let us now introduce some notation. We say that $A \subset X^*$ \emph{$\lambda$-norms} $B \subset X$, for some $\lambda\in (0,1)$, if for every $x \in B$ there is $x^* \in A \setminus \{0\}$ such that $x^*(x) \geq \lambda \|x^*\|\|x\|$, and that $A$ \emph{norms} $B$  if $A$ $\lambda$-norms $B$ for every $\lambda\in(0,1)$. In addition, for $\lambda\in (0,1]$, we denote
$$B^\lambda:=\bigg\{x^*\in X^*\setminus\{0\}:\sup_{x\in B\setminus\{0\}}\frac{x^*(x)}{\|x^*\|\|x\|}\ge\lambda\bigg\}.$$

\begin{remark}\label{rem: properties of notation}
    Let $X$ be a normed space, $\lambda\in (0,1]$ and $A,B\subset X$.
    \begin{itemize}
        \item[(a)] It is clear that $x^*\in B^\lambda$ if and only if $\mu x^*\in B^\lambda$ for every $\mu\in(0,\infty)$. Moreover
$$B^\lambda=\{f(x)\cdot x:x\in B\setminus\{0\}\}^\lambda,$$
where $f:B\setminus\{0\}\rightarrow(0,\infty)$ is any function.
In particular, when we consider the set $B^\lambda$ we can restrict ourselves only to norm 1 elements.
\item[(b)] If $A\subset B$, then $A^\lambda\subset B^\lambda$.
    \end{itemize}
\end{remark}

    \begin{lemma}\label{lem: properties of notation}
Let $X$ be a normed space, $\lambda,\mu\in(0,1]$ and $x,y\in X$. The following implications hold:
\[
\{x\}^\lambda \cap \{y\}^\mu \neq \emptyset \implies x \in (\{y\}^\mu)^\lambda   \implies \n{x+y} \geq \lambda \n x + \mu \n y.
\]
In addition, if $x_1, \dots, x_n \in S_X$ and $\n{\sum_{i=1}^{n} x_i} \geq n-\e$ for some $\varepsilon\in[0,1)$, then there is $\du x \in \du X$ such that $x_i \in \{\du x\}^{(1-\e)} \subset (\{x_j\}^{(1-\e)})^{(1-\e)}$ for all $i,j\in\{1,\ldots,n\}$.
\end{lemma}
\begin{proof}
The first implication is clear. If $x \in (\{y\}^\mu)^\lambda$, then 
\[
\n{x+y} \geq \sup_{\du x \in \{y\}^\mu \cap S_{\du X}} \du x(x+y) \geq \lambda \n x + \mu \n y.
\]
Now let $x_1,\dots, x_n\in S_X$ be such that $\n{\sum_{i=1}^{n} x_i} \geq n-\e$ for some $\varepsilon\in[0,1)$. Fix $\du x \in \left\{\sum_{i=1}^{n} x_i\right\}^1 \cap S_{\du X}$, then
$$n-1+x^*(x_j)\geq \du x(\sum_{i=1}^{n} x_i)=\|\sum_{i=1}^{n} x_i\| \geq n-\e$$
so that $\du x(x_j) \geq 1 - \e$ for all $j\in\{1,\ldots,n\}$. This, together with Remark~\ref{rem: properties of notation}(b), implies that $x_i \in \{\du x\}^{(1-\e)} \subset (\{x_j\}^{(1-\e)})^{(1-\e)}$ for all $i,j\in\{1,\ldots,n\}$.
\end{proof}

\begin{proposition}\label{prop: characterization kappa OH and fail -1 BCP}
  Let $X$ be a normed space and $\kappa$ an infinite cardinal. Consider the following statements:
  \begin{itemize}
      \item[(a)] For every $A\in\mathcal P_{<\kappa}(S_X)$ and $\varepsilon>0$ there is $y\in S_X$ such that
        $$\|x+y\|\ge2-\varepsilon\text{ for every }x\in A;$$
      \item[(a')] $\mathcal P_{<\kappa}(S_X)\subset\bigcap_{\lambda\in(0,1)}\bigcup_{x\in X}\mathcal P((\{x\}^\lambda)^\lambda)$;
      \item[(a'')] $X$ is $(<\kappa)$-octahedral;
      \item[(b)] For every $A\in\mathcal P_{<\kappa}(S_X)$ there is $y\in S_X$ such that
        $$\|x+y\|=2\text{ for every }x\in A;$$
      \item[(b')] $\mathcal P_{<\kappa}(S_X)\subset\bigcup_{x\in X}\mathcal P((\{x\}^1)^1)$;
      \item[(b'')] $X$ fails the $(-1)$-BCP$_{<\kappa}$.
  \end{itemize}
  Then (a)$\iff$(a')$\iff$(a'')$\impliedby$(b)$\iff$(b')$\iff$(b'').
\end{proposition}
\begin{proof}
(a)$\implies$(a'). Let $A\in\mathcal P_{<\kappa}(S_X)$ and $\varepsilon\in(0,1)$. By hypothesis there exists $y\in S_X$ such that $\|x+y\|\ge2-\varepsilon$ for every $x\in A$. Lemma~\ref{lem: properties of notation} implies that $x\in(\{y\}^{(1-\varepsilon)})^{(1-\varepsilon)}$ for every $x\in A$, that is $A\subset (\{y\}^{(1-\varepsilon)})^{(1-\varepsilon)}$.

(a')$\implies$(a). Let $A\in\mathcal P_{<\kappa}(S_X)$ and $\varepsilon\in(0,1)$. By assumption, there is $y\in X\setminus\{0\}$ such that $A\subset(\{y\}^{(1-\varepsilon/2)})^{(1-\varepsilon/2)}$. By Remark~\ref{rem: properties of notation}(a), we can assume that $y\in S_X$, thus Lemma~\ref{lem: properties of notation} shows that $\|x+y\|\ge2-\varepsilon$ for every $x\in A$.

(a'')$\implies$(a) is obvious. (a')$\implies$(a''). Fix a subspace $Y\subset X$ with density character $<\kappa$ and $\varepsilon>0$. Find a set $A$ of cardinality $<\kappa$ dense in $Y$ and define
$$\tilde A:=\{y/\|y\|:y\in (A\cup(-A))\setminus\{0\}\}\in\mathcal{P}_{<\kappa}(S_X).$$
By assumption, there is $x\in S_X$ such that for every $y/\|y\|\in \tilde A$ we can find $x^*_y\in S_{X^*}$ satisfying
$$x^*_y(x)\ge1-\varepsilon\text{ and }x^*_y(y)\ge(1-\varepsilon)\|y\|.$$
Suppose that $\lambda\ge0$. Then, for every $y\in A$,
$$\|y+\lambda x\|\ge x^*_y(y)+\lambda x^*_y(x)\ge(1-\varepsilon)(\|y\|+\lambda)=(1-\varepsilon)(\|y\|+|\lambda|).$$
If $\lambda<0$, then, for every $y\in A$,
$$\|y+\lambda x\|=\|-y-\lambda x\|\ge x^*_{-y}(-y)-\lambda x^*_{-y}(x)\ge(1-\varepsilon)(\|y\|+|\lambda|).$$

The proof of (b)$\iff$(b')$\iff$(b'') is identical and (b)$\implies$(a) is obvious.
\end{proof}

It is known that that the converse of (b)$\implies$(a) in Proposition~\ref{prop: characterization kappa OH and fail -1 BCP} does not hold in general \cite[Theorem~5.13]{ciaci_characterization_2021}. A similar characterization to the one given in Proposition~\ref{prop: characterization kappa OH and fail -1 BCP} also holds for the (A)SD2P. First we need some preliminary results.

    \begin{lemma}\label{lem: SD2P=infinite SD2P}
        Let $X$ be a normed space and $\mathfrak S$ a countable collection of slices of $B_X$. Consider the following statements:
        \begin{itemize}
            \item[(a)] For every $\lambda\in(0,1)$ there is a set $A\subset B_X$ that visits every slice of $\mathfrak S$ and there is $x^*\in S_{X^*}$ such that $x^*(x)\ge\lambda$ for every $x\in A$;
            \item[(b)] Every finite convex combination of slices of $\mathfrak S$ intersects $B_X\setminus\lambda B_X$ for every $\lambda\in(0,1)$;
            \item[(b')] Every convex series of slices of $\mathfrak S$ intersects $B_X\setminus\lambda B_X$ for every $\lambda\in(0,1)$.
        \end{itemize}
        Then (a)$\implies$(b)$\iff$(b'). Moreover, if $\mathfrak S$ is finite, then (b)$\implies$(a).
    \end{lemma}
    \begin{proof}
        (a)$\implies$(b). Fix $S_1,\ldots,S_n\in\mathfrak S$, $\lambda_1,\ldots,\lambda_n\in[0,1]$ such that $\sum_{i=1}^{n}\lambda_i=1$ and $\lambda\in(0,1)$. Find a set $A=\{x_1,\ldots,x_n\}\subset S_X$ and $x^*\in S_{X^*}$ such that $x_i\in S_i$ and $x^*(x_i)\ge\lambda$ for every $i\in\{1,\ldots,n\}$. Thus
        $$\|\sum_{i=1}^{n}\lambda_ix_i\|\ge x^*(\sum_{i=1}^{n}\lambda_ix_i)\ge\lambda.$$
        
        (b)$\implies$(b'). Fix a convex series of slices $\sum_{i=1}^{\infty}\lambda_i S_i$ from $\mathfrak S$ and $\varepsilon>0$. Find $n\in\mathbb N$ such that $\sum_{i=n+1}^\infty\lambda_i<\varepsilon/3$ and $x_1,\ldots,x_n,y\in B_X$ such that
        $$\|\sum_{i=1}^n\lambda_ix_i+(\sum_{i=n+1}^\infty\lambda_i)y\|\ge1-\varepsilon/3\text{ and }x_i\in S_i\text{ for }i\in\{1,\ldots,n\}.$$
        Fix $x_i\in S_i$ for every $i\ge n+1$, hence
        $$\|\sum_{i=1}^{\infty}\lambda_ix_i\|\ge\|\sum_{i=1}^n\lambda_ix_i\|-\varepsilon/3\ge\|\sum_{i=1}^n\lambda_ix_i+(\sum_{i=n+1}^\infty\lambda_i)y\|-2\varepsilon/3\ge 1-\varepsilon.$$
        
        (b')$\implies$(b) is obvious. For the moreover part, assume that $\mathfrak S=\{S_1,\ldots,S_n\}$ and fix $\varepsilon>0$. By our hypothesis we can find $x_i\in S_i\cap S_X$ for every $i\in\{1,\ldots,n\}$ such that $\|\sum_{i=1}^{n} n^{-1}x_i\|\ge1-n^{-1}\varepsilon$, that is $\|\sum_{i=1}^{n} x_i\|\ge n-\varepsilon$. Thanks to Lemma~\ref{lem: properties of notation}, there is $x^*\in S_{X^*}$ such that $x^*(x_i)\ge1-\varepsilon$ for every $i\in\{1,\ldots,n\}$.
    \end{proof}
    
    \begin{lemma}\label{lem: characterization CS}
        Let $X$ be a normed space and $\mathfrak S$ a finite (respectively, countable) collection of slices of $B_X$. The following are equivalent:
        \begin{itemize}
            \item[(i)] There is a set $A\subset B_X$ that visits every slice of $\mathfrak S$ and there is $x^*\in S_{X^*}$ such that $x^*(x)=\|x\|$ for every $x\in A$;
            \item[(ii)] Every finite convex combination (respectively, convex series) of slices in $\mathfrak S$ intersects $S_X$.
        \end{itemize}
    \end{lemma}    
    \begin{proof}
        We will prove only the case when $\mathfrak S$ is countable since the proof when $\mathfrak S$ is finite is identical.
        
        (i)$\implies$(ii). Fix a convex series of slices $\sum_{i=1}^{\infty}\lambda_i S_i$ of $\mathfrak S$. By hypothesis there are a sequence $(x_i)_i\subset S_X$ and $x^*\in S_{X^*}$ such that $x_i\in S_i$ and $x^*(x_i)=1$ for every $i\in\mathbb N$. It is clear that $\|\sum_{i=1}^{\infty}\lambda_ix_i\|\ge x^*(\sum_{i=1}^{\infty}\lambda_ix_i)=1$.
        
        (ii)$\implies$(i). Let $\mathfrak S=\{S_i:i\in\mathbb N\}$. For every $i\in\mathbb N$ we can find $x_i\in S_i\cap S_X$ such that $1=\|\sum_{i=1}^{\infty} 2^{-i}x_i\|=\sum_{i=1}^{\infty}2^{-i}\|x_i\|$. Therefore, 
        $$\|\sum_{i=1}^n x_i\|=\sum_{i=1}^n\|x_i\|$$
        holds for every $n\in\mathbb N$ and Proposition~\ref{prop: norm attaining on norm additive} provides the claim.
    \end{proof}
    
    Finally we can characterize the (A)SD2P without mentioning convex combinations. Let us recall their definitions.
    
    \begin{definition}\label{def: cs sd2p}
        Let $X$ be a normed space and $Y$ a subspace of $X^*$. We say that the couple $(X,Y)$ has the \emph{strong diameter two property} (SD2P for short) if every finite convex combination of slices in $B_X$ defined by functionals in $Y$ intersects $B_X\setminus\lambda B_X$ for every $\lambda\in(0,1)$. 
        
        We say that the couple $(X,Y)$ has the \emph{attaining strong diameter two property} (ASD2P for short) if every finite convex combination of slices in $B_X$ defined by functionals in $Y$ intersects the unit sphere.
    \end{definition}
    
    The original definitions of the (A)SD2P were not given in terms of pairs, but only the pairs $(X,X^*)$ and $(X^*,X)$ were considered separately, nevertheless it is more convenient in the following to consider pairs in order to avoid to repeat each statement twice.
    
    Notice that the original definition of the SD2P, as the name suggests, was given in terms of diameter two convex combinations of slices, but, as already noted in \cite[Theorem~3.1]{lopez-perez_strong_2019}, the definition is in fact equivalent to the one given in Definition~\ref{def: cs sd2p}, so that it becomes obvious that the ASD2P implies the SD2P.
    
    \begin{proposition}\label{prop: characterization CS and SD2P}
      Let $X$ be a normed space and $Y$ a subspace of $X^*$. Consider the following statements:
      \begin{itemize}
          \item[(a)] $(X,Y)$ has the SD2P;
          \item[(a')] For every finite set $A\subset S_Y$ and $\lambda\in(0,1)$ there are $B\subset S_X$ and $x^*\in S_{X^*}$ such that $B$ $\lambda$-norms $A$ and $x^*(x)\ge\lambda$ for every $x\in B$;
          \item[(a'')] $\mathcal P_{<\omega}(S_Y)\subset\bigcap_{\lambda\in(0,1)}\bigcup_{x^*\in X^*}\mathcal{P}((\{x^*\}^\lambda\cap X)^\lambda)$;
          \item[(b)] $(X,Y)$ has the ASD2P;
          \item[(b')] For every finite set $A\subset S_Y$ and $\lambda\in(0,1)$ there are $B\subset S_X$ and $x^*\in S_{X^*}$ such that $B$ $\lambda$-norms $A$ and $x^*(x)=1$ for every $x\in B$;
          \item[(b'')] $\mathcal P_{<\omega}(S_Y)\subset\bigcap_{\lambda\in(0,1)}\bigcup_{x^*\in X^*}\mathcal{P}((\{x^*\}^1\cap X)^\lambda)$.
      \end{itemize}
      Then (a)$\iff$(a')$\iff$(a'')$\impliedby$(b)$\iff$(b')$\iff$(b'').
    \end{proposition}
    \begin{proof}
    (a)$\iff$(a') follows from Lemma~\ref{lem: SD2P=infinite SD2P} and (a') is only a restatement of the more concise (a'').
    
    (b)$\iff$(b') follows from Lemma~\ref{lem: characterization CS}, (b')$\iff$(b'') is identical to (a')$\iff$(a'') and (b)$\implies$(a) is obvious.
    \end{proof}

       It is known that the converse to (b)$\implies$(a) in Proposition~\ref{prop: characterization CS and SD2P} does not hold in general \cite[Example~3.3]{lopez-perez_strong_2019}.

       The most natural way to extend the (A)SD2P to their countable analogues is to consider convex series instead of finite convex combinations of slices. Lemma~\ref{lem: SD2P=infinite SD2P} shows that, for the SD2P, this approach doesn't generate any new concept. On the other hand, Proposition~\ref{prop: characterization CS and SD2P} suggests a new approach that doesn't involve convex combinations and that, as a byproduct, is meaningful for uncountable cardinals too.
    
    \begin{definition}\label{def: kappa-SD2P}
    Let $X$ be a normed space, $Y$ a subspace of $X^*$ and $\kappa$ an infinite cardinal. We say that the couple $(X,Y)$ has the \emph{strong diameter two property with respect to $\kappa$} (SD2P$_\kappa$ for short) if
    $$\mathcal P_\kappa(S_Y)\subset\bigcap_{\lambda\in(0,1)}\bigcup_{x^*\in X^*}\mathcal{P}((\{x^*\}^\lambda\cap X)^\lambda),$$
    or more explicitly, for every set $A\subset S_Y$ of cardinality at most $\kappa$ and $\lambda\in(0,1)$ there are $B\subset S_X$ and $x^*\in S_{X^*}$ such that $B$ $\lambda$-norms $A$ and $x^*(x)\ge\lambda$ for every $x\in B$.
    
    For the couple $(X, X^\ast)$ we simply say that $X$ has the SD2P$_\kappa$ and for the couple $(X^\ast, X)$ we say that $X^*$ has the weak$^*$ SD2P$_\kappa$.
\end{definition}

\begin{definition}\label{def: kappa-CS}
     Let $X$ be a normed space, $Y$ a subspace of $X^*$ and $\kappa$ an infinite cardinal. We say that the couple $(X,Y)$ has the \emph{1-norming attaining strong diameter two property with respect to $\kappa$} (1-ASD2P$_\kappa$ for short) if
     $$\mathcal P_\kappa(S_Y)\subset\bigcup_{x^*\in X^*}\mathcal{P}((\{x^*\}^1\cap X)^1),$$
     or more explicitly, for every set $A\subset S_Y$ of cardinality at most $\kappa$ there are $B\subset S_X$ and $x^*\in S_{X^*}$ such that $B$ norms $A$ and $x^*(x)=1$ for every $x\in B$.
     
     For the couple $(X, X^\ast)$ we simply say that $X$ has the 1-ASD2P$_\kappa$ and for the couple $(X^\ast, X)$ we say that $X^*$ has the weak$^*$ 1-ASD2P$_\kappa$.
     \end{definition}

Notice that $X$ can have at most the SD2P$_{<|X^*|}$.

\begin{remark}\label{rem: def ASD2P}
~
\begin{itemize}
    \item[(a)] In Definitions~\ref{def: kappa-SD2P} and \ref{def: kappa-CS} we can replace $S_Y$ with $Y\setminus\{0\}$. In fact, given $A\subset Y\setminus\{0\}$ and $B\subset S_X$, $B$ $\lambda$-norms $A$ if and only if $B$ $\lambda$-norms the set $\{y/\|y\|:y\in A\}$, for some $\lambda\in (0,1)$.
    \item[(b)] Thanks to Lemma~\ref{lem: characterization CS}, the 1-ASD2P$_\omega$ corresponds to the property that every convex series of slices in the unit ball intersects the unit sphere. The same, due to Lemma~\ref{lem: SD2P=infinite SD2P}, can't be said about the SD2P$_\omega$.
    \item[(c)] Notice that the 1-ASD2P$_{<\omega}$ and ASD2P are actually distinct properties. In fact $c_0$ has the ASD2P \cite[Theorem~2.4]{abrahamsen_relatively_2018}, but it fails the 1-ASD2P$_{<\omega}$ due to Proposition~\ref{prop: X CSP and X sep OH}. On the other hand, the SD2P$_{<\omega}$ as in Definition~\ref{def: kappa-SD2P} naturally corresponds to the original SD2P as in Definition~\ref{def: cs sd2p} thanks to Proposition~\ref{prop: characterization CS and SD2P}.
\end{itemize}
\end{remark}

\begin{remark}
    It is natural also to consider the concept of \emph{1-norming strong diameter 2 property with respect to $\kappa$} given by requiring in Definition~\ref{def: kappa-SD2P} that $B$ norms $A$ and of \emph{attaining strong diameter 2 property with respect to $\kappa$} given by requiring in Definition~\ref{def: kappa-CS} that $B$ $\lambda$-norms $A$ for some fixed $\lambda\in(0,1)$. Many results contained in the following sections can be adapted to these definitions too, nevertheless in this paper we are not going to investigate these two properties in detail.
\end{remark}

\section{Duality}\label{sec: duality}

Recall that a Banach space $X$ has the SD2P (respectively, $X$ is octahedral) if and only if $X^*$ is octahedral (respectively, $X^*$ has the weak$^*$ SD2P) (see \cite[Remark~II.5.2]{godefroy_metric_1989} and \cite[Corollary~2.2]{becerra_guerrero_octahedral_2014}).

Our next aim is to show that similar statements also hold for bigger cardinals.

\begin{lemma}\label{lem: X SD2P iff X* OH}
    If $X$ is a Banach space, then
    $$\bigcap_{\lambda\in(0,1)}\bigcup_{x^*\in X^*}\mathcal P((\{x^*\}^\lambda)^\lambda\cap X^*)=\bigcap_{\lambda\in(0,1)}\bigcup_{x^*\in X^*}\mathcal P((\{x^*\}^\lambda\cap X)^\lambda).$$
\end{lemma}
\begin{proof}
    Fix $A\in\bigcap_{\lambda\in(0,1)}\bigcup_{x^*\in X^*}\mathcal P((\{x^*\}^\lambda)^\lambda\cap X^*)$. For every $\varepsilon>0$ there is $y^{*}\in S_{X^{*}}$ such that $A\subset(\{y^{*}\}^{(1-\varepsilon/2)})^{(1-\varepsilon/2)}\cap X^*$. Namely, for every $x^*\in A$ we can find $x^{**}\in S_{X^{**}}$ such that
  $$x^{**}(x^*)\ge1-\varepsilon/2\text{ and }x^{**}(y^*)\ge1-\varepsilon/2.$$
  By Goldstine's theorem, there is $x\in S_X$ such that
  $$|(x^{**}-x)(x^*)|<\varepsilon/2\text{ and }|(x^{**}-x)(y^*)|<\varepsilon/2,$$
  therefore $x^*(x),y^*(x)>1-\varepsilon.$ In other words $A\subset(\{y^{*}\}^{(1-\varepsilon)}\cap X)^{(1-\varepsilon)}$. We conclude that
  $$\bigcap_{\lambda\in(0,1)}\bigcup_{x^*\in X^*}\mathcal P((\{x^*\}^\lambda)^\lambda\cap X^*)\subset\bigcap_{\lambda\in(0,1)}\bigcup_{x^*\in X^*}\mathcal P((\{x^*\}^\lambda\cap X)^\lambda).$$
  The opposite inclusion is obvious since $(\{x^*\}^\lambda\cap X)^\lambda\subset (\{x^*\}^\lambda)^\lambda\cap X^*$ for every $x^*\in X^*$ and $\lambda\in(0,1)$.
\end{proof}

\begin{theorem}\label{thm: sd2p iff oh}
    Let $X$ be a Banach space and $\kappa$ an infinite cardinal. The following statements hold:
    \begin{itemize}
        \item[(a)] $X$ has the SD2P$_{<\kappa}$ if and only if $X^*$ is $(<\kappa)$-octahedral;
        \item[(b)] $X^*$ has the weak$^*$ SD2P$_{<\kappa}$ if and only if for every $A\in\mathcal P_{<\kappa}(S_X)$ and $\varepsilon>0$ there is $x^{**}\in S_{X^{**}}$ such that $\|x+x^{**}\|\ge2-\varepsilon\text{ for every }x\in A$.
    \end{itemize}
\end{theorem}
\begin{proof}
    The conclusion follows from Proposition~\ref{prop: characterization kappa OH and fail -1 BCP} combined with Lemma~\ref{lem: X SD2P iff X* OH}.
\end{proof}

\begin{remark}\label{rem: classical duality}
Note that by taking $\kappa=\omega$ in Theorem~\ref{thm: sd2p iff oh}  (and also applying the Principle of Local Refelexivity in (b)) we recover the known classical dualities between the SD2P and octahedral norms. 
\end{remark}

We now turn our attention to the 1-ASD2P$_{<\kappa}$ and its dual connections to the failure of the $(-1)$-BCP$_{<\kappa}$.

\begin{proposition}\label{prop: X sep OH and X* CSP}
Let $X$ be a Banach space and $\kappa$ an infinite cardinal. Consider the following statements:
\begin{itemize}
    \item[(a)] $X$ fails the $(-1)$-BCP$_{<\kappa}$;
    \item[(b)] $X^{\ast}$ has the weak$^\ast$ 1-ASD2P$_{<\kappa}$;
    \item[(c)] For every $A \in\mathcal P_{<\kappa}(S_X)$ there is $x^{**} \in S_X^{**}$
    such that $\n{x + x^{**}}=2$ for every $x \in A$;
    \item[(d)] $X^*$ has the weak$^*$ SD2P$_{<\kappa}$.
\end{itemize}
Then (a)$\Rightarrow$(b)$\Rightarrow$(c)$\Rightarrow$(d).
\end{proposition}
\begin{proof}
(a)$\Rightarrow$(b) follows from Proposition~\ref{prop: characterization kappa OH and fail -1 BCP}. (b)$\Rightarrow$(c). Fix $A \in\mathcal P_{<\kappa}(S_X)$. Find $B\subset S_{X^*}$ and $x^{**}\in S_{X^{**}}$ such that $B$ norms $A$ and $x^{**}(x^*)=1$ for every $x^*\in B$. Then for every $x\in A$
$$\|x+x^{**}\|\ge\sup_{x^*\in B}x^*(x)+x^{**}(x^*)=2.$$
(c)$\Rightarrow$(d) follows from Theorem~\ref{thm: sd2p iff oh}(b).
\end{proof}

\begin{remark}
    Notice that the converse of (a)$\Rightarrow$(b) in Proposition~\ref{prop: X sep OH and X* CSP} does not hold in general since we will show that $\ell_1(\kappa)^*$ has the weak$^\ast$ 1-ASD2P$_\kappa$ (Proposition~\ref{prop: ell_infty has w*-CS}), but $\ell_1(\kappa)$ clearly has the $(-1)$-BCP$_\kappa$. A partial positive vice-versa to (c)$\Rightarrow$(d) will be given in Theorem~\ref{thm: w* SD2P=w* w SD2P}.
\end{remark}

\begin{proposition}\label{prop: X CSP and X sep OH}
Let $X$ be a normed space and $\kappa$ an infinite cardinal. If $X$ has the 1-ASD2P$_{<\kappa}$, then $X^\ast$ fails the $(-1)$-BCP$_{<\kappa}$.
\end{proposition}
\begin{proof}
It follows from Proposition~\ref{prop: characterization kappa OH and fail -1 BCP}.
\end{proof}

However, we do not know whether the converse of Proposition~\ref{prop: X CSP and X sep OH} holds.

\begin{ques}
    Let $X$ be a Banach space and $\kappa$ an infinite cardinal. If $X^*$ fails the $(-1)$-BCP$_{<\kappa}$, then does $X$ have the 1-ASD2P$_{<\kappa}$?
\end{ques}

\begin{remark}\label{rem: extend cs to all cardinals}
    Let $X$ be a normed space and $Y\subset X^*$. Notice that a set $A\subset S_X$ that norms a dense set in $S_Y$ also norms $S_Y$ itself. Therefore we conclude that if $(X,Y)$ has the 1-ASD2P$_{\text{dens}(Y)}$, then $(X,Y)$ has the 1-ASD2P with respect to the maximum meaningful cardinality, that is $(X,Y)$ has the 1-ASD2P$_{|Y|}$.
\end{remark}

Now we focus on some results that hold when $\kappa=\omega$.

\begin{proposition}\label{prop: dual of Daugavet has w*-CS_infty}
Let $X$ be a Banach space. If $X$ has the Daugavet property, then $X^*$ has the weak$^*$ 1-ASD2P$_\omega$. If in addition $X$ is separable, then $X^*$ has the weak$^*$ 1-ASD2P$_{|X|}$.
\end{proposition}
\begin{proof}
    The argument in \cite[Example~3.4]{abrahamsen_relatively_2018} actually proves that $X^*$ has the weak$^*$ 1-ASD2P$_\omega$. Remark~\ref{rem: extend cs to all cardinals} provides the second part of the claim.
\end{proof}

We will show later on that the SD2P and the SD2P$_\omega$ are different properties (see Example~\ref{ex: c_0 fails infty CS}), nevertheless the same cannot be said for its weak$^*$ analogue.
 
 \begin{theorem}\label{thm: w* SD2P=w* w SD2P}
 Let $X$ be a Banach space. The following assertions are equivalent:
 \begin{itemize}
     \item[(i)] $X$ is octahedral;
     \item[(ii)] for every separable subspace $Y\subset X$ there exists $x^{\ast\ast}\in S_{X^{\ast\ast}}$ such that 
\begin{equation*}
  \| y + \lambda x^{\ast\ast} \| =\|y\|+|\lambda|
  \ \mbox{for every } y \in Y\mbox{ and }\lambda \in \mathbb{R};
\end{equation*}
     \item[(iii)] $X^\ast$ has the weak$^\ast$ SD2P;
     \item[(iv)] $X^\ast$ has the weak$^\ast$ SD2P$_\omega$.
 \end{itemize}
 \end{theorem}
 \begin{proof}
    (i)$\implies$(ii). Let $Y$ be a separable subspace of $X$. By \cite[Proposition~3.36]{langemets_geometrical_2015}, there is a separable and octahedral subspace $Z$ of $X$ such that $Y\subset Z\subset X$. From \cite[Lemma~9.1]{godefroy_ball_1989} we can assure the existence of $z^{\ast\ast}\in S_{Z^{\ast\ast}}\subset S_{X^{\ast\ast}}$ such that 
 \begin{equation*}
  \| z + \lambda z^{\ast\ast} \| =\|z\|+|\lambda|
  \ \mbox{for every } z \in Z\mbox{ and }\lambda \in \mathbb{R}.
\end{equation*}
(ii)$\implies$(iv) follows from Theorem~\ref{thm: sd2p iff oh}(b) and Remark~\ref{rem: classical duality}, (iv)$\implies$(iii) is obvious and (iii)$\iff$(i) is \cite[Remark~II.5.2]{godefroy_metric_1989} or \cite[Theorem~2.1]{becerra_guerrero_octahedral_2014}.
 \end{proof}

\section{Direct sums}\label{sec: direct sums}

Recall that a norm $N$ on $\mathbb R^2$ is said to be \emph{absolute} if $N(a,b)=N(|a|,|b|)$ holds for every $(a,b)\in\mathbb R^2$ and \emph{normalized} if $N(1,0)=N(0,1)=1$ (see \cite{Bonsall1973}). For example, every $\ell_p$-norm, for $1\le p\le\infty$, is absolute and normalized. 

Given two normed spaces $X$ and $Y$, and an absolute norm $N$, we denote by $X\oplus_NY$ the product $X\times Y$ endowed with the norm $\|(x,y)\|_N:=N(\|x\|,\|y\|)$.

If $N$ is an absolute normalized norm on $\mathbb R^2$, then the \emph{dual norm} $N^*$ is defined by
$$N^*(c,d):=\max_{(a,b)\in B_{(\mathbb R^2,N)}}(|ac|+|bd|)$$
for every $(c,d)\in\mathbb R^2$. One can routinely verify that $N^*$ is an absolute normalized norm and $(X\oplus_NY)^*=X^*\oplus_{N^*}Y^*$. Lastly, we recall that an element $(a,b)\in\mathbb R^2$ is called \emph{positive} if $a,b\ge0$.

\begin{definition}\label{def: positive sd2p}(\cite{haller_stability_2017})
We say that an absolute normalized norm $N$ on $\mathbb R^2$ has the \emph{positive SD2P} if for every finite family of slices $S_1,\ldots,S_n$ in $B_{(\mathbb R^2,N)}$, defined by positive functionals in $(\mathbb R^2,N^*)$, there are positive $(a_i,b_i)\in S_i$ such that $N(\sum_{i=1}^{n}(a_i,b_i))=\sum_{i=1}^{n} N(a_i,b_i)$.
\end{definition}

The original definition of the positive SD2P was given in terms of convex combination of slices, but Lemma~\ref{lem: properties additive sets}(a) ensures that it is equivalent to Definition~\ref{def: positive sd2p}. It is known that an absolute normalized norm $N$ on $\mathbb R^2$ has the positive SD2P if and only if there is a positive $(a,b)$ such that
$$N(t(a,1)+(1-t)(b,1))=1\text{ for every }t\in[0,1],$$
hence it is clear that the $\ell_1$ and the $\ell_\infty$-norm both have the positive SD2P, but they are not the only such norms.

The positive SD2P was introduced in order to characterize all the absolute normalized norms that preserve the SD2P \cite[Theorem~3.5]{haller_stability_2017}. In addition, it was noted in \cite[Proposition~2.2]{haller_convex_2017} that the same statement holds also for the ASD2P. In the following we want to prove that both these statements can in fact be extended to infinite cardinals and dual pairs.

\begin{lemma}\label{lem: positive SD2P}
    Let $N$ be an absolute normalized norm on $\mathbb R^2$. The following are equivalent:
    \begin{itemize}
        \item[(i)] $N$ has the positive SD2P;
        \item[(ii)] For every family of slices $(S_\eta)_\eta$ in $B_{(\mathbb R^2,N)}$, defined by positive functionals in $(\mathbb R^2,N^*)$, there are positive $(a_\eta,b_\eta)\in S_\eta$ satisfying
        $$N(\sum_{i=1}^{n}(a_{\eta_i},b_{\eta_i}))=\sum_{i=1}^{n} N(a_{\eta_i},b_{\eta_i})\text{ for every }\eta_1,\ldots,\eta_n;$$
        \item[(iii)] For every set  $A\subset S_{(\mathbb R^2,N)^*}$ consisting of positive functionals there is a set $B\subset S_{(\mathbb R^2,N)}$ consisting of positive elements and $x^*\in S_{(\mathbb R^2,N^*)}$ such that $B$ norms $A$ and $x^*(x)=1$ for every $x\in B$.
    \end{itemize}
\end{lemma}
\begin{proof}
    (ii)$\implies$(i) is obvious and (ii)$\iff$(iii) follows from Proposition~\ref{prop: norm attaining on norm additive}.
    
    (i)$\implies$(ii). Suppose at first that we are given countably many slices $\{S_i:i\in\mathbb N\}$ in $B_{(\mathbb R^2,N)}$ defined by positive functionals. For every $n\in\mathbb N$ there are positive elements $(a_1^n,b_1^n),\ldots,(a_n^n,b_n^n)$ satisfying $(a_i^n,b_i^n)\in S_i$ for $i\in\{1,\ldots,n\}$ and $N(\sum_{i=1}^{n}(a_i^n,b_i^n))=\sum_{i=1}^{n}N(a_i^n,b_i^n)$. We proceed with a diagonal argument.
    
    Thanks to compactness, there are $I_1\subset\mathbb N$ and $(a_1,b_1)\in S_1$ such that $\lim_{i\in I_1}(a_1^i,b_1^i)=(a_1,b_1)$. Suppose we are given $I_n$, then we can find $I_{n+1}\subset I_n$ and $(a_{n+1},b_{n+1})\in S_{n+1}$ such that $\lim_{i\in I_{n+1}}(a_{n+1}^i,b_{n+1}^i)=(a_{n+1},b_{n+1})$. It is clear that for every $n\in\mathbb N$
    $$\lim_{i\in I_n}(a_j^i,b_j^i)=(a_j,b_j)\text{ for every }j\in\{1,\ldots,n\}.$$
    Now, the sequence $\{(a_i,b_i):i\in\mathbb N\}$ satisfies the desired properties, because, for every $n\in\mathbb N$,
    $$N(\sum_{j=1}^n(a_j,b_j))=\lim_{i\in I_n}N(\sum_{j=1}^n(a_j^i,b_j^i))=\lim_{i\in I_n}\sum_{j=1}^nN(a_j^i,b_j^i)=\sum_{j=1}^nN(a_j,b_j).$$
    Hence the implication (i)$\implies$(ii)$\iff$(iii) is proved whenever the family of slices is countable. Since $S_{(\mathbb R^2,N^*)}$ is separable, an argument similar to the one in Remark~\ref{rem: extend cs to all cardinals} shows that that claim holds for families of slices of any cardinality.
\end{proof}

\begin{proposition}\label{prop: positive SD2P}
  Let $X_1$ and $X_2$ be normed spaces, $Y_1\subset X_1^*$ and $Y_2\subset X_2^*$ non-trivial subspaces, $\kappa$ an infinite cardinal and $N$ an absolute normalized norm on $\mathbb R^2$. Let $P$ be one of the following properties: SD2P$_{<\kappa}$, 1-ASD2P$_{<\kappa}$ or ASD2P. If $(X_1,Y_1)$ and $(X_2,Y_2)$ satisfy $P$, then $N$ has the positive SD2P if and only if $(X_1\oplus_NX_2,Y_1\oplus_{N^*}Y_2)$ satisfies $P$.
\end{proposition}
\begin{proof}
We are going to prove the claim only for the SD2P$_{<\kappa}$ since the proof for the 1-ASD2P$_{<\kappa}$ and for the ASD2P is almost identical.

    Suppose that $N$ has the positive SD2P. Fix $A\in\mathcal P_{<\kappa}(S_{Y_1\oplus_{N^*}Y_2})$ and $\varepsilon\in(0,1)$. Call $A_1:=\{y_1/\|y_1\|:(y_1,y_2)\in A\text{ and }y_1\not=0\}\subset S_{Y_1}$ and $A_2:=\{y_2/\|y_2\|:(y_1,y_2)\in A\text{ and }y_2\not=0\}\subset S_{Y_2}$. Since $(X_i,Y_i)$ has the SD2P$_{<\kappa}$, we can find $B_i\subset S_{X_i}$ and $x^*_i\in S_{X_i^*}$ such that
    $$B_i\ (1-\varepsilon/2)^{1/2}\text{-norms }A_i\text{ and }x^*_i(x)\ge1-\varepsilon\text{ for every }x\in B_i.$$
    Thanks to Lemma~\ref{lem: positive SD2P}, we can find $B\subset S_{(\mathbb R^2,N)}$ consisting of positive elements and $(c,d)\in S_{(\mathbb R^2,N^*)}$ such that
    $$B\text{ norms }\{(\|y_1\|,\|y_2\|):(y_1,y_2)\in A\}\text{ and }ac+bd=1\text{ for every }(a,b)\in B.$$
    Define $z^*:=(cx_1^*,dx_2^*)\in S_{(X_1\oplus_NX_2)^*}$ and
    $$\tilde B:=\{(ax_1,bx_2):(a,b)\in B,x_1\in B_1\text{ and }x_2\in B_2\}\subset S_{X_1\oplus_NX_2}.$$
    It is clear that, for every $(ax_1,bx_2)\in\tilde B$,
    $$z^*(ax_1,bx_2)=acx^*_1(x_1)+bdx_2^*(x_2)\ge(1-\varepsilon)(ac+bd)=1-\varepsilon.$$
    We claim that $\tilde B$ $(1-\varepsilon)$-norms $A$. Fix $(y_1,y_2)\in A$, find $x_i\in B_i$ and $(a,b)\in B$ such that $y_i(x_i)\ge(1-\varepsilon)^{1/2}\|y_i\|$ for $i\in\{1,2\}$ and $a\|y_1\|+b\|y_2\|\ge(1-\varepsilon)^{1/2}$. Thus $(y_1,y_2)(ax_1,bx_2)\ge(1-\varepsilon)^{1/2}(a\|y_1\|+b\|y_2\|)\ge1-\varepsilon$.
    
    Vice-versa assume that $(X_1\oplus_NX_2,Y_1\oplus_{N^*}Y_2)$ has the the SD2P$_{<\kappa}$, hence the SD2P. It is easy to adapt the proof of \cite[Theorem~3.5]{haller_stability_2017} to show that this implies that $N$ has the positive SD2P.
\end{proof}

Now we focus on $\ell_1$ and $\ell_\infty$-sums.

\begin{lemma}\label{lem: absolute sums 1}
Let $X_1$ and $X_2$ be normed spaces, $Y_1 \subset \du X_1$ and $Y_2 \subset \du X_2$ subspaces, $\lambda\in (0,1]$ and $\e\geq 0$.
The following statements hold:
\begin{enumerate}[(a)]
\item
If $A \subset S_{X_1}$ and $B \subset S_{X_2}$, then $A^\lambda\oplus_1B^\lambda\subset (A\oplus_\infty B)^\lambda$;
\item
If $A \subset S_{X_1 \oplus_\infty X_2}$, then 
\[
A^{(1-\e)} \cap (S_{Y_1} \oplus_1 S_{Y_2}) \subset p_{X_1}(A)^{(1-2\e)}\oplus_1 p_{X_2}(A)^{(1-2\e)};
\]
\item
If $x^* \in X^*_1 \setminus \{0\}$, then
$$(\{x^*\}^\lambda \cap S_{X_1}) \oplus_\infty B_{X_2} \subset \{x^*\oplus_10\}^\lambda\cap S_{X_1 \oplus_\infty X_2};$$ 
\item
If $x_1^* \oplus_1 x_2^* \in S_{X_1^* \oplus_1 X_2^*}$ and $A \subset  \{x_1^* \oplus_1 x_2^*\}^{(1-\e)} \cap S_{X_1 \oplus_\infty X_2}$, then either $p_{X_1}(A) \subset \{x_1^*\}^{(1-2\e)}\cap X_1$ or $p_{X_2}(A) \subset \{x_2^*\}^{(1-2\e)}\cap X_2$.
\end{enumerate}

\end{lemma}
\begin{proof}
(a), (b) and (c) are obvious.
(d). Assume by contradiction (this includes the case $x_2^*=0$) that we can find $x_1\in p_{X_1}(A)$ and $y_2\in p_{X_2}(A)$ such that
$$\du x_1 (x_1)<(1-2\varepsilon)\|x_1^*\|\|x_1\|\text{ and }x_2^*(y_2)\leq(1-2\varepsilon)\|x_2^*\|\|y_2\|.$$
Choose $y_1$ and $x_2$ such that $(x_1,y_1),(x_2,y_2)\in A$, thus
\begin{align*}
    2-2\varepsilon &\le (x_1^*,x_2^*)((x_1,y_1)+(x_2,y_2))\\
    &<(1-2\varepsilon)(\|x_1^*\|\|x_1\|+\|x_2^*\|\|y_2\|)+x_1^*(x_2)+x_2^*(y_1)\\
    &\le1-2\varepsilon+x_1^*(x_2)+x_2^*(y_1),
\end{align*}
hence $x_1^*(x_2)+x_2^*(y_1)>1$, which is a contradiction since $(x_1^*,x_2^*)\in S_{X_1^*\oplus_1X_2^*}$ and $(x_2, y_1)\in B_{ X_1\oplus_\infty X_2}$.
\end{proof}

\begin{lemma}\label{lem: absolute sums 2}
  Let $X_1$ and $X_2$ be normed spaces, $Y_1 \subset \du X_1$ and $Y_2 \subset \du X_2$ subspaces, $\lambda$ and $\mu\in (0,1]$, $\e$ and $\delta\geq 0$.
The following statements hold:
\begin{itemize}
    \item[(a)]
    If $x^* \in X^*_1 \setminus \{0\}$, then
    $$(\{x^*\}^{\lambda} \cap X_1)^\mu \oplus_1 \du{X_2} \subset (\{x^*\oplus_1 0\}^\lambda \cap (X_1 \oplus_\infty X_2))^\mu.$$
    \item[(b)]
    Let $A \subset S_{Y_1}$, $B \subset S_{Y_2}$ and $x^*_1 \oplus_1 x_2^* \in S_{X^*_1 \oplus_1 X^*_2}$. If $A \oplus_1 B \subset (\{x_1^* \oplus_1 x_2^*\}^{(1-\e)} \cap (X_1 \oplus_\infty X_2))^{(1-\delta)}$, then either
    $$A \subset (\{x^*_1\}^{(1-2\e)} \cap X_1)^{(1-2\delta)}\text{ or }B\subset (\{x^*_2\}^{(1-2\e)} \cap X_2)^{(1-2\delta)}$$
\end{itemize}
\end{lemma}
\begin{proof}
(a). Using Remark~\ref{rem: properties of notation}(b), Lemma~\ref{lem: absolute sums 1}(a) and (c), we obtain that
\begin{align*}
    (\{x^*\}^{\lambda} \cap X_1)^\mu \oplus_1 \du{X_2} &= (\{x^*\}^{\lambda} \cap S_{X_1})^\mu \oplus_1 (S_{X_2})^\mu\\
    &\subset((\{x^*\}^{\lambda} \cap S_{X_1}) \oplus_\infty S_{X_2})^\mu\\
    &\subset (\{x^* \oplus_1 0\}^\lambda \cap (X_1 \oplus_\infty X_2))^\mu.
\end{align*}

(b). If $C := \{x_1^* \oplus_1 x_2^*\}^{(1-\e)} \cap S_{X_1 \oplus_\infty X_2}$, then Remark~\ref{rem: properties of notation}(a) together with Lemma~\ref{lem: absolute sums 1}(b) implies that
$$A\oplus_1B\subset C^{(1-\delta)}\cap(S_{Y_1}\oplus_1 S_{Y_2})\subset p_{X_1}(C)^{(1-2\delta)}\oplus_1 p_{X_2}(C)^{(1-2\delta)},$$
hence $A \subset p_{X_1}(C)^{(1-2\delta)}$ and $B \subset p_{X_2}(C)^{(1-2\delta)}$.
On the other hand, Lemma~\ref{lem: absolute sums 1}(d) gives that either $p_{X_1}(C) \subset \{x_1^*\}^{(1-2\e)}\cap X_1$
or $p_{X_2}(C) \subset \{x_2^*\}^{(1-2\e)}\cap X_2$. The conclusion follows from Remark~\ref{rem: properties of notation}(b).
\end{proof}

\begin{proposition}\label{prop: infty sum for weak*CS_infty}
    Let $X_1$ and $X_2$ be normed spaces, $Y_1\subset \du X_1$ and $Y_2 \subset \du X_2$ subspaces and $\kappa$ an infinite cardinal. Let $P$ be one of the following properties: SD2P$_{<\kappa}$, 1-ASD2P$_{<\kappa}$ or ASD2P. The following are equivalent:
    \begin{itemize}[label=(\alph*)]
    \item[(i)]
    \label{item.comphasws}
     Either $(X_1,Y_1)$ or $(X_2,Y_2)$ satisfy $P$;
    \item[(ii)]
    \label{item.o1hasws}
    $(X_1\oplus_\infty X_2,Y_1\oplus_1 Y_2)$ satisfies $P$.
    \end{itemize}
\end{proposition}

\begin{proof}
(i)$\implies$(ii) follows from (a) and (ii)$\implies$(i) from (b) of Lemma~\ref{lem: absolute sums 2}.
\end{proof}

\begin{remark}
    Notice that Lemma~\ref{lem: absolute sums 2} can be exploited for proving stability results for $\kappa$-octahedrality and for the failure of the $(-1)$-BCP$_\kappa$ too.
\end{remark}

\begin{proposition}\label{prop: CS_infty 1-sum}
  Let $(X_\eta)_\eta$ be a family of normed spaces, $Y_\eta\subset X^*_\eta$ subspaces and $\kappa$ an infinite cardinal. Let $P$ be one of the following properties: SD2P$_{<\kappa}$, 1-ASD2P$_{<\kappa}$ or ASD2P. The following are equivalent:
  \begin{itemize}
      \item[(i)] $(X_\eta,Y_\eta)$ satisfies $P$ for every $\eta$;
      \item[(ii)] $(\ell_1(X_\eta),\ell_\infty(Y_\eta))$ satisfies $P$.
  \end{itemize}
\end{proposition}
\begin{proof}
We are going to prove the claim only for the SD2P$_{<\kappa}$ since the proofs for the 1-ASD2P$_{<\kappa}$ and for the ASD2P are almost identical.

  (i)$\implies$(ii). Fix $A \in\mathcal P_{<\kappa}(S_{\ell_\infty(Y_\eta)})$ and $\lambda\in(0,1)$. For every $\eta$ find $C_\eta\subset S_{X_\eta}$ and $x^*_\eta\in S_{X^*_\eta}$ such that $C_\eta$ $\lambda^{1/2}$-norms $p_{Y_\eta}(A)$ and $x^*_\eta (x)\ge\lambda$ for every $x\in C_\eta$.
  Set $C:=\bigcup_\eta i_{X_\eta}(C_\eta)\subset S_{\ell_1(X_\eta)}$ and define $x^*\in S_{\ell_\infty(X^*_\eta)}$ by $x^*(\eta):=x^*_\eta$. It is clear that $x^*(x)\ge\lambda$ for every $x\in C$. We claim that $C$ $\lambda$-norms $A$. For any $\du y \in A$ find $\eta$ such that $\n{\du y(\eta)} \geq \lambda^{1/2}$ and $x_\eta \in C_\eta$ satisfying $\du y(\eta)(x_\eta) \geq \lambda^{1/2} \n{\du y(\eta)}$.
  We conclude that $\du y (i_{X_\eta} x_\eta) =\du y(\eta) (x_\eta) \geq \lambda$.
  
  (ii)$\implies$(i). It is enough to prove that $(X_1,Y_1)$ has the SD2P$_{<\kappa}$ whenever $(X_1 \oplus_1 X_2,Y_1\oplus_\infty Y_2)$ has it. Fix $ A \in\mathcal P_{<\kappa}(S_{Y_1})$ and $\varepsilon\in(0,1)$. There is $x^*_1\oplus_\infty x^*_2\in S_{X^*_1\oplus_\infty X^*_2}$ such that, if we call $C := \{\du x_1 \oplus_\infty \du x_2\}^{(1-\e)} \cap (X_1 \oplus_1 X_2)$, then $A \oplus_\infty \{0\} \subset C^{(1-\varepsilon)}$. It is clear that $A \subset p_{X_1}(C)^{(1-\varepsilon)}$. On the other hand, $p_{X_1}(C) \cap S_{X_1} \subset \{\du x_1\}^{(1-2\e)}$ due to Lemma~\ref{lem: absolute sums 1}(b). We conclude that $A\subset(\{x_1^*\}^{(1-2\varepsilon)})^{(1-\varepsilon)}$ thanks to Remark~\ref{rem: properties of notation}.

  \end{proof}
  
  We conclude this section by noting that any infinite $\ell_\infty$-sum of dual spaces always has the weak$^*$ 1-ASD2P with respect to the maximum meaningful cardinality.

\begin{proposition}\label{prop: ell_infty has w*-CS}
  If $(X_\eta)_\eta$ is an infinite family of non-trivial normed spaces, then $\ell_1(X_\eta)^*$ has the weak$^*$ 1-ASD2P$_{|\ell_1(X_\eta)|}$.
\end{proposition}
\begin{proof}
Find $y^* \in S_{\ell_\infty(X_\eta^*)}$ such that $\n{y^*(\eta)} = 1$ for every $\eta$ and set $A:=(y^* + c_{00}(X_\eta^*))\cap S_{\ell_\infty(X^*_\eta)}$. We claim that $A$ norms $\ell_1(X_\eta)$.

Fix $x\in S_{\ell_1(X_\eta)}$ and $\e > 0$. Find a finite set $I$ such that $\sum_{\eta\notin I}\|x(\eta)\|<\varepsilon/2$. Pick $x^*\in S_{\ell_\infty(X_\eta^*)}$ such that $x^*(x) = 1$ and define $\tilde x^* \in A$ by
  $$\tilde x^*(\eta):=\bigg\{
  \begin{array}{ll}
    x^*(\eta)&\text{if }\eta\in I,\\
    y^*(\eta)&\text{if }\eta\notin I. 
  \end{array}$$
  Therefore,
  $$\tilde x^*(x)=\sum_\eta \tilde x^*(\eta)(x(\eta))\ge\sum_{\eta\in I}x^*(\eta)(x(\eta))-\varepsilon/2\ge x^*(x)-\varepsilon = 1-\varepsilon,$$
  hence the claim. In order to conclude, we only need to find $x^{**}\in S_{X^{**}}$ such that $x^{**}(x^*)=1$ for every $x^*\in A$. For every $x^*_1, \dots, x^*_n\in A$, there is $\eta$ such that $x^*_i(\eta) = y^*(\eta)$ for all $i\in\{1, \dots, n\}$. Thus, $\|\sum_{i=1}^{n} x^*_i\|_\infty\ge\|\sum_{i=1}^{n} y^*(\eta)\|=n$ and the conclusion follows from Proposition~\ref{prop: norm attaining on norm additive}.
\end{proof}

\section{Tensor products}\label{sec: tensor products}

Given two Banach spaces $X$ and $Y$, we will denote by $X\pten Y$ the projective and by $X\iten Y$ the injective tensor product of $X$ and $Y$. Recall that the space $\mathcal B(X\times Y )$ of bounded bilinear forms defined on $X\times Y$ is isometrically isomorphic to the topological dual of $X\pten Y$. We refer to \cite{Ryan2002a} for a detailed treatment and applications of tensor products.

It is known that the projective tensor product preserves both the SD2P \cite[Corollary~3.6]{becerra_guerrero_octahedral_2015} and the ASD2P \cite[Proposition~3.6]{lopez-perez_strong_2019}. We begin this section by extending this result to the 1-ASD2P$_\kappa$.

The following lemma implies that, in the definition of the 1-ASD2P$_\kappa$, we can replace the request of finding a set that norms a fixed set of functionals with the condition that the the set visits a fixed set of relatively weakly open sets in $B_X$. 

\begin{lemma}\label{lem: CS for RWO}
    Let $X$ be a normed space, $\lambda\in(0,1]$ and $\kappa$ an infinite cardinal. The following are equivalent:
    \begin{itemize}
        \item[(i)] For every family of slices $\mathfrak S$ of $B_X$ of cardinality strictly smaller than $\kappa$ there are $A\subset S_X$ and $x^*\in S_{X^*}$ such that $A$ visits every slice of $\mathfrak S$ and $x^*(x)\ge\lambda$ for every $x\in A$;
        \item[(ii)] For every family of non-empty relatively weakly open sets $\mathfrak U$ in $B_X$ of cardinality strictly smaller than $\kappa$ there are $A\subset S_X$ and $x^*\in S_{X^*}$ such that $A$ visits every set of $\mathfrak U$ and $x^*(x)\ge\lambda$ for every $x\in A$.
    \end{itemize}
\end{lemma}
\begin{proof}
    (ii)$\implies$(i) is obvious since every slice is a non-empty relatively weakly open subset of the unit ball.
    
  (i)$\implies$(ii). Fix $\mathfrak U$ as in the assumption. Thanks to Bourgain's lemma \cite[Lemma~II.1]{Ghoussoub1987b}, for every $U\in\mathfrak U$ we can find some finite convex combination of slices $\sum_{i=1}^{n_U}\lambda_{i,U}S_{i,U}\subset U$. There is a set $A=\{x_{i,U}:i\in\{1,\ldots,n_U\},U\in\mathfrak U\}\subset S_X$ such that $x_{i,U}\in S_{i,U}$ and there is $x^*\in S_{X^*}$ such that $x^*(x_{i,U})\ge\lambda$ for every $i\in\{1,\ldots,n_U\}$ and $U\in\mathfrak U$. The set $\{\sum_{i=1}^{n_U}\lambda_{i,U}x_{i,U}:U\in\mathfrak U\}$ visits every set of $\mathfrak U$ and $x^*(\sum_{i=1}^{n_U}\lambda_{i,U}x_{i,U})\ge\lambda$ for every $U\in\mathfrak U$.
\end{proof}

 \begin{proposition}\label{prop: CS in proj tensor product}
	Let $X$ and $Y$ be Banach spaces and $\kappa$ an infinite cardinal. If $X$ and $Y$ have the 1-ASD2P$_{<\kappa}$, then $X\pten Y$ has the 1-ASD2P$_{<\kappa}$.
\end{proposition}
 \begin{proof}
 Let $(B_\eta)_\eta\subset S_{(X\pten Y)^*}$ be a family of cardinality $<\kappa$ consisting of bilinear forms and choose elements $u_{\eta,n}\otimes v_{\eta,n}\in S_X\otimes S_Y$ such that $B_\eta(u_{\eta,n},v_{\eta,n})>(1-1/n)^{1/2}$ for every $\eta$ and $n\in\mathbb N$. Since $X$ has the 1-ASD2P$_\kappa$, there are $x_{\eta,n}\subset S_X$ and $\xs \in S_{\Xs}$ such that $B_\eta(x_{\eta,n},v_{\eta,n})\ge(1-1/n)^{1/2}\|B_\eta(\cdot,v_{\eta,n})\|$ and $\xs(x_{\eta,n})=1$ for every $\eta$ and $n\in\mathbb N$. Consider the relatively weakly open sets 
	\[
	V_{\eta,n}:=\{y\in B_Y\colon B_\eta(x_{\eta,n}, y)>(1-1/n)^{1/2}\|B_\eta(\cdot,v_{\eta,n})\|\}.
	\]
Since $Y$ has the 1-ASD2P$_\kappa$, by Lemma~\ref{lem: CS for RWO}, there are $y_{\eta,n}\in V_{\eta,n}$ and $y^*\in S_{Y^*}$ such that $y^*(y_{\eta,n})=1$. Set $z_{\eta,n}:=x_{\eta,n}\otimes y_{\eta,n}$ and $z^*:=x^*\otimes y^*$. Clearly, $z^*(z_{\eta,n})=1$ and $\{z_{\eta,n}:\eta,n\}$ norms $(B_\eta)_\eta$ since
	\begin{align*}
	B_\eta(x_{\eta,n},y_{\eta,n})>(1-1/n)^{1/2}\|B_\eta(\cdot,v_{\eta,n})\|\ge1-1/n.
	\end{align*}
\end{proof}

\begin{remark}\label{rem: proj ten CS one component is not enough}
	Proposition~\ref{prop: CS in proj tensor product} remains no longer true in general if one assumes that only $X$ has the 1-ASD2P$_\kappa$. Indeed, by \cite[Corollary~3.9]{langemets_octahedral_2017}, the space $\ell_\infty\pten \ell^3_3$ fails the SD2P although we will later prove that $\ell_\infty$ has the 1-ASD2P$_{<2^{2^\omega}}$ (Corollary~\ref{cor: ell infty C[0,1] CS}).
\end{remark}

We now investigate sufficient conditions for the space of bounded linear operators to be $\kappa$-octahedral. As an application we get that the SD2P$_{\kappa}$ is stable under projective tensor products (Corollary~\ref{cor: SD2P_k in proj ten}) and that $\kappa$-octahedral norms behave well under injective tensor products (Corollary~\ref{cor: k_OH in inj ten}). 

We begin by introducing a weakening of $\kappa$-octahedrality which is inspired by \cite[Definition~2.1]{haller_rough_2017}.

\begin{definition}
Let $X$ be a normed space and $\kappa$ an infinite cardinal. We say that $X$ is \emph{alternatively $\kappa$-octahedral} if, whenever $A\in \mathcal P_\kappa(S_X)$ and $\eps>0$, there is a $y\in S_X$ such that
\[
\max\{\|x+y\|,\|x-y\| \}\geq2-\varepsilon\quad\text{ for all $x\in A$}.
\]
\end{definition}

Note that the alternative $\kappa$-octahedrality of $X$ is equivalent to the following conditions:
\begin{itemize}
\item[(i)]
whenever $A\in\mathcal P_\kappa(S_X)$ and $\lambda\in(0,1)$, there is $y\in S_X$ such that for every $x\in A$ there is $x^*\in S_{X^*}$ satisfying
\[
|x^*(x)|\ge\lambda
\quad\text{and}\quad
|x^*(y)|\ge\lambda;
\]
\item[(ii)] $\mathcal P_{<\kappa}(S_X)\subset\bigcap_{\lambda\in(0,1)}\bigcup_{x\in X}\mathcal P((\{x\}^\lambda\cup\{-x\}^\lambda)^\lambda\cup(-(\{x\}^\lambda\cup\{-x\}^\lambda)^\lambda)).$
\end{itemize}

Observe that $\kappa$-octahedral spaces are alternatively $\kappa$-octahedral, but in general the converse fails. In fact it is clear that $\ell_\infty(\kappa)$ is alternatively $\kappa$-octahedral, but it is not even octahedral \cite[Example~5.5]{ciaci_characterization_2021}.

\begin{proposition}\label{prop: kappa-OH operator spaces}

Let $X$ and $Y$ be Banach spaces, $H$ a closed subspace of $\lxy$ containing the finite rank operators, and $\kappa$ an infinite cardinal.
\begin{itemize}
\item[{\rm(a)}]
If  $\Xs$ is $(<\kappa)$-octahedral and $Y$ is alternatively $(<\kappa)$-octahedral, then  $H$ is $(<\kappa)$-octahedral.

\item[{\rm(b)}]
If $\Xs$ is alternatively $(<\kappa)$-octahedral and $Y$ is $(<\kappa)$-octahedral, then $H$ is $(<\kappa)$-octahedral.

\end{itemize}

\end{proposition}

\begin{proof}
The proof is an adaptation of the one given for octahedral spaces in \cite[Theorem~2.2]{haller_rough_2017}.
\end{proof}

Recall that $(X\pten Y)^\ast$ is also isometrically isomorphic to $\mathcal{L}(X,Y^\ast)$, therefore Proposition~\ref{prop: kappa-OH operator spaces} combined with Theorem~\ref{thm: sd2p iff oh} immediately gives the following.

\begin{corollary}\label{cor: SD2P_k in proj ten}
Let $X$ and $Y$ be Banach spaces and $\kappa$ an infinite cardinal. If $X$ and $Y$ have the SD2P$_{<\kappa}$, then $X\pten Y$ has the SD2P$_{<\kappa}$.
\end{corollary}

Note that, by Remark~\ref{rem: proj ten CS one component is not enough}, it is clear that for the SD2P$_{<\kappa}$ of $X\pten Y$ it is not sufficient to assume the SD2P$_{<\kappa}$ only from one of the components $X$ or $Y$. 

\begin{proposition}\label{prop: sep OH in L(X*,Y)}
  Let $X$ and $Y$ be Banach spaces,
  $H$ a subspace of $\mathcal{L}(X^*,Y)$ containing $X \otimes Y$
  such that every $T \in H$ is weak$^*$-weak continuous, and $\kappa$ an infinite cardinal.
  If the norm of $X$ is alternatively $(<\kappa)$-octahedral and the norm of $Y$
  is $(<\kappa)$-octahedral, then the norm of $H$ is $(<\kappa)$-octahedral.
\end{proposition}

\begin{proof}
The proof is an adaptation of the one given for octahedral spaces in \cite[Proposition~3.1]{langemets_octahedral_2017}.
\end{proof}

Since $X \iten Y$ is a subspace of $L(X^*,Y)$ consisting of weak$^\ast$-weak continuous functions, then Proposition~\ref{prop: sep OH in L(X*,Y)} implies the following.

\begin{corollary}\label{cor: k_OH in inj ten}
Let $X$ and $Y$ be Banach spaces and $\kappa$ an infinite cardinal. If $X$ and $Y$ are $(<\kappa)$-octahedral, then  $X\iten Y$ is $(<\kappa)$-octahedral.
\end{corollary}

\section{Examples}\label{sec: examples}

\subsection{Elementary examples}
We begin by providing some simple examples of Banach spaces having the 1-ASD2P$_\kappa$.

\begin{lemma}\label{lem: sufficient condition for infty-CS}
    Let $X$ be a normed space and $\kappa$ an infinite cardinal. If for every set $A \in\mathcal P_\kappa(S_{\du X})$ there are $B \subset S_X$ and $y\in S_X$ such that $B$ norms $A$ and $y \pm B \subset S_X$, then $X$ has the 1-ASD2P$_\kappa$.
\end{lemma}
\begin{proof}
    The proof is an adaptation of \cite[Example~3.3]{abrahamsen_relatively_2018}.
\end{proof}

Let $\eta<\kappa$ be infinite cardinals. In the following we denote by $\ell^\eta_\infty(\kappa)$ the elements in $\ell_\infty(\kappa)$ with support of size at most $\eta$.

\begin{example}\label{ex: c_0(kappa) has CS}
    Let $\omega\le\eta<\kappa$. If either $X=c_0(\kappa)$ or $\kappa$ is a regular cardinal and $X=\ell^\eta_\infty(\kappa)$, then $X$ satisfies the hypothesis of Lemma~\ref{lem: sufficient condition for infty-CS}, hence it has the 1-ASD2P$_{<\kappa}$. Indeed, for every set $A\in\mathcal P_{<\kappa}(S_X)$, we can find $\eta\in \kappa\setminus \bigcup_{x\in A} \text{supp}\{x\}$ such that $\|x\pm e_\eta\|=1$ for every $x\in A$.
\end{example}

\subsection{$C(K)$ and $L_1(\mu)$ spaces}

We first investigate when Banach spaces of the form $C(K)$, where $K$ is compact and Hausdorff, have the 1-ASD2P$_\kappa$. Recall that if $K$ is a compact Hausdorff space, then it is normal and $C(K)^*$ can be identified with the space of regular signed Borel measures of bounded variation \cite[Theorem~14.14]{aliprantis_infinite_2006}.

\begin{theorem}\label{thm: C(K) has CS}
    Let $K$ be a compact Hausdorff space. If $|K|\ge\omega_1$, then $C(K)$ has the 1-ASD2P$_{<|K|}$.
\end{theorem}
\begin{proof}
  Fix a family $(\mu_\eta)_\eta\subset S_{C(K)^*}$ of cardinality $<|K|$. Since each measure $\mu_\eta$ has bounded variation, it has at most countably many atoms, hence we can find $x\in K\setminus(\bigcup_\eta\{\text{atoms of }\mu_\eta\})$. Fix $\eta$ and $\varepsilon>0$. Since $\mu_\eta(\{x\})=0$ and $\mu_\eta$ is regular, we can find an open set $U\subset K$ such that $x\in U$ and $|\mu_\eta(U)|\le\varepsilon/3$. There is $f\in S_{C(K)}$ such that $\mu_\eta(f)\ge1-\varepsilon/3$ and consider $F:K\rightarrow[0,1]$ the Urysohn's function such that $F(x)=0$ and $F=1$ in $K\setminus U$. Set $f_{\eta,\varepsilon}:=1+F(f-1)\in S_{C(K)}$. Notice that
  $$\mu_\eta(f_{\eta,\varepsilon})=\int_{K\setminus U}fd\mu_\eta+\int_Uf_{\eta,\varepsilon}d\mu_\eta\ge \mu_\eta(f)-2|\mu_\eta(U)|\ge1-\varepsilon.$$
  This means that $\{f_{\eta,\varepsilon}:\eta,\varepsilon\}$ norms $(\mu_\eta)_\eta$. In addition, $\delta_x(f_{\eta,\varepsilon})=1$ for every $\eta$ and $\varepsilon$.
\end{proof}

\begin{corollary}\label{cor: ell infty C[0,1] CS}
If $X$ is either $C[0,1]$ or $\ell_\infty$, then $X$ has the 1-ASD2P$_{<|X^*|}$.
\end{corollary}
\begin{proof}
    Recall that $\ell_\infty=C(\beta\mathbb N)$, where $\beta\mathbb N$ is the Stone--\v{C}ech compactification of $\mathbb N$. As $\beta\mathbb N$ is the set of all ultrafilters on $\omega$, by \cite[Theorem~7.6]{Jech}, we have $|\beta\mathbb N|=2^{2^\omega}$. We claim that $|\ell_\infty^*|=2^{2^\omega}$, hence the conclusions follows from Theorem~\ref{thm: C(K) has CS}. In fact
    $$|\ell_\infty^*|\le|\mathbb{R}|^{|\ell_\infty|}=(2^\omega)^{2^\omega}=2^{\omega\cdot2^\omega}=2^{2^\omega},$$
    while the reverse inequality is obvious since $\delta_x\in\ell_\infty^*$ for every $x\in \beta\mathbb N$. The argument for $C[0,1]$ is similar since $|[0,1]|=2^\omega$ and $|C[0,1]^*|=2^\omega$. In fact, let $A\subset C[0,1]$ be some countable dense subset, then
    $$|C[0,1]^*|\le|\mathbb R|^{|A|}=(2^\omega)^\omega=2^{\omega\cdot\omega}=2^\omega,$$
    while the reverse inequality is trivial since $\delta_x\in C[0,1]^*$ for every $x\in[0,1]$.
\end{proof}

\begin{corollary}\label{cor: C(K) has CS}
    Let $K$ be compact Hausdorff space. The following are equivalent:
    \begin{itemize}
        \item[(i)] $|K|\ge\omega_1$;
        \item[(ii)] $C(K)$ has the SD2P$_\omega$;
        \item[(iii)] $C(K)$ has the 1-ASD2P$_\omega$;
        \item[(iv)] $C(K)^*$ is non-separable;
        \item[(v)] $C(K)^*$ is $\omega$-octahedral;
        \item[(vi)] $C(K)^*$ fails the $(-1)$-BCP$_\omega$.
    \end{itemize}
\end{corollary}
\begin{proof}
  (iv)$\iff$(vi) is \cite[Corollary~20]{guirao_remarks_2019}, (ii)$\iff$(v) is Theorem~\ref{thm: sd2p iff oh}(a) and (i)$\implies$(iii) follows from Theorem~\ref{thm: C(K) has CS}. (iii)$\implies$(ii) and (vi)$\implies$(v)$\implies$(iv) are obvious. Notice that (iv)$\implies$(i) since if $|K|<\omega_1$, then $C(K)^*=\overline{\text{span}}\{\delta_x:x\in K\}$.
\end{proof}

\begin{corollary}\label{cor: L_infty has CS}
    $L_\infty[0,1]$ has the 1-ASD2P$_\omega$.
\end{corollary}
\begin{proof}
    There exists a compact Hausdorff space $K$ such that $L_\infty[0,1]$ is isometrically isomorphic to the space $C(K)$ \cite[Theorem~4.2.5]{Albiac2006b}. Since $L_\infty[0,1]$ is non-separable, then its dual must be non-separable, hence Corollary~\ref{cor: C(K) has CS} implies the claim.
\end{proof}

Let $(\Omega,\Sigma,\mu)$ be a measure space. We now proceed to investigate when $L_1(\mu)$ has the SD2P$_\omega$ or the 1-ASD2P$_\omega$.

\begin{theorem}\label{thm: L-embedded}
    Let $X$ be an L-embedded space and $\kappa$ an infinite cardinal. Let $P$ be one of the following properties: SD2P$_{<\kappa}$, 1-ASD2P$_{<\kappa}$ or ASD2P. Then $X$ satisfies $P$ if and only if $X^{**}$ satisfies the weak$^*$ $P$. 
\end{theorem}
\begin{proof}
If $X$ satisfies $P$, then clearly $X^{**}$ has the corresponding weak$^*$ property.

Conversely, suppose now that $X$ is $L$-embedded, that is, $X^{**}=X\oplus_1Z$ for some subspace $Z\subset X^{**}$. It is clear that $X^*=X^*\oplus_\infty \{0\}$, therefore Proposition~\ref{prop: CS_infty 1-sum} implies that, if $(X\oplus_1Z,X^*\oplus_\infty\{0\})$ satisfies $P$, then $(X,X^*)$ satisfies $P$.
\end{proof}

Recall that $L_1(\mu)^*=L_\infty(\mu)$ if and only if $\mu$ is localizable.

\begin{corollary}\label{cor: L_1(mu) CS_omega}
    Let $(\Omega, \Sigma, \mu)$ be a localizable measure space. The following statements are equivalent:
    \begin{itemize}
    \item[(i)] $\mu$ is atomless;
    \item[(ii)] $L_1(\mu)$ has the SD2P;
    \item[(iii)] $L_1(\mu)$ has the 1-ASD2P$_\omega$;
    \item[(iv)] $B_{L_1(\mu)}$ has no strongly exposed points;
    \item[(v)] $L_1(\mu)$ has the Daugavet property;
    \item[(vi)] $L_\infty(\mu)$ is octahedral;
    \item[(vii)] $L_\infty(\mu)$ fails the $(-1)$-BCP$_\omega$;
    \item[(viii)] $L_\infty(\mu)$ has the Daugavet property.
    \end{itemize}
\end{corollary}
\begin{proof}
    (i)$\iff$(iv)$\iff$(v) was proved in \cite{guerrero_daugavet_2006} and (i)$\iff$(viii) is known (see e.g. \cite{werner_recent_2001}). (viii)$\implies$(iii) follows from combining the fact that $L_1(\mu)$ is an $L$-embedded space \cite[IV Example~1.1]{Harmand1993c} together with Proposition~\ref{prop: dual of Daugavet has w*-CS_infty} and Theorem~\ref{thm: L-embedded}. (iii)$\implies$(ii) and (ii)$\implies$(iv) are obvious, (ii)$\iff$(vi) follows from Theorem~\ref{thm: sd2p iff oh} and (iii)$\implies$(vii) is shown by Proposition~\ref{prop: X CSP and X sep OH}. Eventually, (vii)$\implies$(vi) is obvious.
\end{proof}

\begin{remark}\label{rem: daugavet+separable is not CS}
The examples of separable Banach spaces that have the 1-ASD2P$_\omega$ shown so far all have the Daugavet property. Nevertheless, they still are separate properties. In fact, fix some separable Daugavet spaces $X$ and $Y$ that also have the 1-ASD2P$_\omega$ (for example, $X=Y=C[0,1]$) and let $N$ be an absolute normalized norm with the postive SD2P which differs from the $\ell_1$ and $\ell_\infty$-norm, then $X\oplus_NY$ has the 1-ASD2P$_\omega$ by Proposition~\ref{prop: positive SD2P}. On the other hand, $X\oplus_NY$ is  separable and cannot have the Daugavet property \cite[Corollary~5.4]{bilik_narrow_2005}.
\end{remark}

In \cite[Theorem~4.4]{abrahamsen_remarks_2013} it was proved that every Daugavet space has the SD2P. In the following we will show that even more is true.

\begin{proposition}\label{prop: Daugavet implies SD2P_omega}
  Let $X$ be a Banach space. If $X$ has the Daugavet property, then $X$ has the SD2P$_\omega$.
\end{proposition}
\begin{proof}
If $X$ has the Daugavet property, then, by \cite[Lemma~5.1]{ciaci_characterization_2021}, the dual $X^\ast$ fails $(-1)$-BCP$_\omega$. Thus, $X^*$ is $\omega$-octahedral and by Theorem~\ref{thm: sd2p iff oh} we conclude that $X$ has the SD2P$_\omega$.
\end{proof}

Note that the converse of Proposition~\ref{prop: Daugavet implies SD2P_omega} fails by considering either the space $\ell_\infty$ or Remark~\ref{rem: daugavet+separable is not CS}. However, it seems to be unknown whether the following strengthening is true in general.

\begin{ques}\label{ques: daugavet separable}
    If $X$ is a Banach space with the Daugavet property, then does $X$ have the ASD2P (or even the 1-ASD2P$_\omega$)?
\end{ques}

The answer to Question~\ref{ques: daugavet separable} is positive if either 
\begin{itemize}
    \item[(a)] $X$ is a separable $L$-embedded space. Indeed, if $X$ is separable, $L$-embedded, and has the Daugavet property, then $X^*$ also has the Daugavet property \cite[Theorem~3.4]{zoca_daugavet_2018}. Now, by Proposition~\ref{prop: dual of Daugavet has w*-CS_infty}, the bidual $X^{**}$ has the weak$^*$ 1-ASD2P$_\omega$, hence Theorem~\ref{thm: L-embedded} shows that $X$ has the 1-ASD2P$_\omega$.
    \item[(b)] $X=C(K)$ is separable for some $K$ compact Hausdorff. Since $X$ has the Daugavet property, then $K$ does not have any isolated points \cite[Example~(a)]{werner_recent_2001}. Now, by Corollary~\ref{cor: C(K) has CS}, $X$ has the 1-ASD2P$_\omega$ if and only if $K$ is uncountable. If, by contradiction, we would assume that $K$ was countable, then we would get that $K=\omega^\alpha+n$ for some countable ordinal $\alpha$ and $n\in\mathbb N$ by Sierpinski--Mazurkiewicz theorem, which clearly has isolated points. Therefore, $X$ must have the 1-ASD2P$_\omega$.
\end{itemize}

It is important to point out that there exists a normed space which has the Daugavet property and is strictly convex \cite[Theorem~5.2]{kadets_remarks_1996}, hence it fails the ASD2P \cite[Proposition~2.3]{lopez-perez_strong_2019}. It seems to be unknown whether there exists a strictly convex Banach space with the Daugavet property.

\subsection{Lebesgue--Bochner spaces $L_1(\mu;X)$}
In general it is not known when Lebesgue--Bochner spaces $L_1(\mu;X)$ are $L$-embedded (see \cite[IV.5]{Harmand1993c} for some partial results).
Therefore, more attention is needed to study under which conditions Lebesgue--Bochner spaces $L_1(\mu;X)$ have the SD2P$_\omega$ or the 1-ASD2P$_\omega$.

\begin{lemma}\label{lem: disjoin measurable sets}
Let $(\Omega, \Sigma, \mu)$ be a measure space such that $\mu$ is atomless. If $(B_i)_i\subset\Sigma$ is a sequence of non-negligible sets, then there exists a sequence $(E_i)_i\subset\Sigma$ consisting of non-negligible pairwise disjoint sets satisfying $E_i \subset B_i$ for every $i\in\mathbb N$.
\end{lemma}
\begin{proof}
    For every $i > 1$ fix a set $C_i^1 \subset B_i$ such that $0 < \mu(C_i^1)  \leq 2^{-i} \mu(B_1)$ (the existence of such $C_i^1$'s is ensured by \cite[Theorem~10.52]{aliprantis_infinite_2006}) and set $E_1 := B_1 \setminus (\bigcup_{i=2}^\infty C_i^1)$. Note that
    \[
\mu(E_1) \geq \mu(B_1) - (\sum_{i=2}^\infty 2^{-i})\mu(B_1) > 0.
\]
We now have $E_1 \subset B_1$ and $C_i^1 \subset B_i$, all non-negligible, such that $E_1$ is disjoint from $\bigcup_{i=2}^\infty C_i^1$. With an induction argument, we repeat the construction on $\{C_i^{j-1}\}_{i=j+1}^\infty$ to get $E_j \subset C_j^{j-1}$ disjoint from all $C_i^j \subset C_i^{j-1}$. Clearly, the sequence $(E_i)_i$ satisfies the claim.
\end{proof}

\begin{lemma}\label{lem: L_infty}
Let $X$ be a Banach space and $(\Omega, \Sigma, \mu)$ a measure space such that $\mu$ is atomless. If $(f_i)_i\subset S_{L_\infty(\mu;X)}$, then there exists a  sequence $(E_i)_i\subset \Sigma$ of non-negligible pairwise disjoint sets such that $(f_i \chi_{E_i})_i \subset S_{L_\infty(\mu;X)}$.
\end{lemma}
\begin{proof}
For all $i,j \in \bb N$ find non-negligible $B_{i,j}\in\Sigma$ such that $\|f_i(t)\|>1-1/j$ for almost all $t \in B_{i,j}$. By Lemma~\ref{lem: disjoin measurable sets}, we can assume that these sets are pairwise disjoint. Set $E_i := \bigcup_{j \in \bb N} B_{i,j}$.
\end{proof}

It is known that $L_1(\mu;X)^*=L_\infty(\mu;X^*)$ if either $\mu$ is decomposable and $X^*$ is separable \cite[p.~282]{Dinculeanu_vector} or $\mu$ is $\sigma$-finite and $X^*$ has the Radon--Nikod\'ym property with respect to $\mu$ \cite[p.~98]{Diestel_vector}. 

\begin{theorem}\label{thm: L_1(mu) has CS_infty}
  Let $(\Omega, \Sigma, \mu)$ be a measure space and $X$ a Banach space. Suppose that either $\mu$ is decomposable and $X^*$ is separable or $\mu$ is $\sigma$-finite and $X^*$ has the Radon--Nikod\'ym property with respect to $\mu$. 
	\begin{enumerate}
		\item[(a)]
	If $\mu$ is atomless, then $L_1(\mu;X)$ has the 1-ASD2P$_\omega$.
	  \item[(b)]
	Let $P$ denote any of the following properties: SD2P, ASD2P,\\ 1-ASD2P$_{<\omega}$, SD2P$_\omega$, or 1-ASD2P$_\omega$.
	If $\mu$ has atoms, then $L_1(\mu;X)$ satisfies P if and only $X$ satisfies $P$.
\end{enumerate}
\end{theorem}
\begin{proof}
$(a)$ Let $(f_i)_i\subset S_{L_\infty(\mu;X)}$. Thanks to Lemma~\ref{lem: L_infty}, there exists a pairwise disjoint sequence $(E_{i,j})_{i,j}\subset \Sigma$ consisting of non-negligible sets such that $(f_i \chi_{E_{i,j}})_{i,j} \subset S_{L_\infty(\mu;X)}$. For every $i,j\in \mathbb{N}$ find some $g_{i,j}\in S_{L_1(\mu;X)}$ such that
$$(f_i\chi_{E_{i,j}})(g_{i,j})\ge1-1/j.$$
It is clear that the set $\{g_{i,j}\chi_{E_{i,j}}:i,j\in\mathbb N\}$ satisfies condition (ii) of Proposition~\ref{prop: norm attaining on norm additive} and it norms the $f_i$'s since
$$f_i(g_{i,j}\chi_{E_{i,j}})=(f_i\chi_{E_{i,j}})(g_{i,j})\ge1-1/j.$$

$(b)$
If $\mu$ has atoms, then
$$L_1(\mu;X)=L_1(\nu;X)\oplus_1\ell_1(\{X:\eta\text{ is an atom of }\mu\}),$$
where $\nu$ is atomless, thus the claim follows from $(a)$ and 
Proposition~\ref{prop: CS_infty 1-sum}.
\end{proof}

\begin{remark}\label{rem: duality L_1 and L_infty}
    Notice that if $X^*$ is separable, then $X$ cannot have the SD2P$_\omega$ due to Theorem~\ref{thm: sd2p iff oh}. On the other hand, if $X^*$ has the Radon--Nikod\'ym property with respect to some $\sigma$- finite measure $\mu$, then $X$ can still have the 1-ASD2P$_\omega$. In fact $c_0(\omega_1)$ has the 1-ASD2P$_\omega$ thanks to Example~\ref{ex: c_0(kappa) has CS} and $\ell_1(\omega_1)$ has the Radon--Nikod\'ym property.
\end{remark}

It is known that if $\mu$ is localizable, then $\mu$ is purely atomic if and only if every finite convex combination of weak$^*$ slices in $B_{L_1(\mu)^*}$ is relatively weak$^*$ open \cite[Theorem~4.1]{lopez-perez_strong_2019} (as already noted, this condition implies the weak$^*$ ASD2P).

We now investigate when the dual of the Lebesgue--Bochner space $L_1(\mu;X)$ has the weak$^*$ SD2P$_\omega$ or the weak$^*$ 1-ASD2P$_\omega$.

\begin{proposition}
  Let $(\Omega, \Sigma, \mu)$ be a measure space and $X$ a Banach space. Let $P$ denote one of the following properties: weak$^*$ SD2P, weak$^*$ ASD2P, weak$^*$ 1-ASD2P$_{<\omega}$, weak$^*$ SD2P$_\omega$ or weak$^*$ 1-ASD2P$_\omega$. Then $L_1(\mu;X)^*$ satisfies $P$ if and only if $\mu$ is not purely atomic with finitely many atoms or $X^*$ satisfies $P$.
\end{proposition}
\begin{proof}
  If $\mu$ is atomless, then $L_1(\mu;X)$ has the Daugavet property \cite[page~81]{werner_recent_2001}.  Therefore, Proposition~\ref{prop: dual of Daugavet has w*-CS_infty} implies that $L_1(\mu;X)^*$ has the weak$^*$ 1-ASD2P$_\omega$. If $\mu$ is purely atomic and has infinitely many atoms, then Proposition~\ref{prop: ell_infty has w*-CS} implies that $L_1(\mu;X)^*$ has the weak$^*$ 1-ASD2P$_\omega$. If $\mu$ has atoms but is not purely atomic, then $L_1(\mu;X)=L_1(\mu_1;X)\oplus_1 L_1(\mu_2;X)$, where $\mu_1$ is atomless, therefore, thanks to Proposition~\ref{prop: infty sum for weak*CS_infty}, we conclude that $L_1(\mu;X)^*$ has the weak$^*$ 1-ASD2P$_\omega$.
  
  If $\mu$ is purely atomic and has finitely many atoms, then 
  $$
  L_1(\mu;X)=X\oplus_1\ldots\oplus_1 X,
  $$ therefore Proposition~\ref{prop: infty sum for weak*CS_infty} implies that $L_1(\mu;X)^*$ satisfies $P$ if and only if $X^*$ satisfies $P$.
\end{proof}

\subsection{Counterexamples}
Let us finish with some more counterexamples in order to observe the following implication diagram.

\begin{figure}[H]
\centering
\begin{tikzpicture}
\path
(0,0) node (1) {weak$^*$ SD2P$_{<\kappa}$}
(0,1.5) node (2) {weak$^*$ 1-ASD2P$_{<\kappa}$}
(2,0.4) node [scale=0.68]{Ex.~\ref{ex: C[0,1] weak* CS}}
(2,1.9) node [scale=0.68]{Ex.~\ref{ex: C[0,1] weak* CS}}
(3.5,0) node (3) {SD2P$_{<\kappa}$}
(3.5,1.5) node (4) {1-ASD2P$_{<\kappa}$}
(4.3,0.75) node [scale=0.68]{Ex.~\ref{ex: SD2P not= CS}}
(5.3,0.4) node [scale=0.68]{Ex.~\ref{ex: c_0 fails infty CS}}
(5.3,1.9) node [scale=0.68]{Ex.~\ref{ex: c_0 fails infty CS}}
(7,0) node (5) {SD2P$_\kappa$}
(7,1.5) node (6) {1-ASD2P$_\kappa$};
\draw[-implies,double equal sign distance]
([xshift=-1mm]4.south)--([xshift=-1mm]3.north);
\draw[nimplies,implies-,double equal sign distance]
([xshift=1mm]4.south)--([xshift=1mm]3.north);
\draw[-implies,double equal sign distance]
(2.south)--(1.north);
\draw[-implies,double equal sign distance]
([yshift=-1mm]6.west)--([yshift=-1mm]4.east);
\draw[nimplies,implies-,double equal sign distance]
([yshift=1mm]6.west)--([yshift=1mm]4.east);
\draw[-implies,double equal sign distance]
([yshift=-1mm]5.west)--([yshift=-1mm]3.east);
\draw[nimplies,implies-,double equal sign distance]
([yshift=1mm]5.west)--([yshift=1mm]3.east);
\draw[-implies,double equal sign distance]
([yshift=-1mm]4.west)--([yshift=-1mm]2.east);
\draw[nimplies,implies-,double equal sign distance]
([yshift=1mm]4.west)--([yshift=1mm]2.east);
\draw[-implies,double equal sign distance]
([yshift=-1mm]3.west)--([yshift=-1mm]1.east);
\draw[nimplies,implies-,double equal sign distance]
([yshift=1mm]3.west)--([yshift=1mm]1.east);
\end{tikzpicture}
\end{figure}

\begin{example}\label{ex: C[0,1] weak* CS}
    $C[0,1]^*$ has the weak$^\ast$ 1-ASD2P$_{2^\omega}$ thanks to Proposition~\ref{prop: dual of Daugavet has w*-CS_infty} (recall that $|C[0,1]|=2^\omega$ since $|C[0,1]|=|C([0,1]\cap\mathbb Q)|\le|\mathbb R|^{\mathbb Q}=(2^\omega)^\omega=2^{\omega\cdot\omega}=2^\omega$ and the opposite inequality is trivial), but it fails the SD2P \cite[Example~1.1]{haller_duality_2015}.
\end{example}

\begin{example}\label{ex: c_0 fails infty CS}
	If $\kappa\ge\omega_1$, then $c_0(\kappa)$ has the 1-ASD2P$_{<\kappa}$ thanks to Example~\ref{ex: c_0(kappa) has CS}, but it fails the SD2P$_\kappa$ due to Theorem~\ref{thm: sd2p iff oh}.
\end{example}

\begin{example}\label{ex: SD2P not= CS}
    It is known that the ASD2P and the SD2P are not equivalent \cite[Example~3.3]{lopez-perez_strong_2019}. The same can be said for the 1-ASD2P$_\kappa$ and the SD2P$_\kappa$. In fact, for every infinite cardinal $\kappa$, there exists a dual Banach space $X^*$ which is $\kappa$-octahedral, but that has the $(-1)$-BCP$_\kappa$ \cite[Theorem~5.13]{ciaci_characterization_2021}. Therefore, thanks to Theorem~\ref{thm: sd2p iff oh}, $X$ has the SD2P$_\kappa$ and it fails the 1-ASD2P$_\kappa$ by Proposition~\ref{prop: X CSP and X sep OH}.
\end{example}

\bibliographystyle{siam}

\end{document}